\documentclass[11pt]{amsart}
\usepackage[T1]{fontenc}
\usepackage[utf8]{inputenc}
\usepackage{amsmath,amssymb,amsthm}
\usepackage{mathtools}
\usepackage{booktabs}
\usepackage{array}
\usepackage[margin=1.1in]{geometry}
\usepackage[hidelinks]{hyperref}

\allowdisplaybreaks

\newcommand{\E}{\mathbb{E}}
\DeclareMathOperator{\Cov}{Cov}
\DeclareMathOperator{\Var}{Var}

\theoremstyle{plain}
\newtheorem{theorem}{Theorem}
\newtheorem{lemma}[theorem]{Lemma}
\newtheorem{proposition}[theorem]{Proposition}
\newtheorem{corollary}[theorem]{Corollary}

\theoremstyle{definition}
\newtheorem{definition}[theorem]{Definition}

\theoremstyle{remark}
\newtheorem{remark}[theorem]{Remark}

\title[On Higher-Order Geometric Refinements of Classical Covariance Asymptotics]{On Higher-Order Geometric Refinements of Classical Covariance Asymptotics: an Approach via Intrinsic and Extrinsic Information Geometry}

\author{Malik Amir}
\address{CRCHUM, Montreal, Canada}
\address{Neuro Connect Lab, CHUSJ, Montreal, Canada}
\address{Soci\'et\'e Qu\'eb\'ecoise de l'Intelligence Artificielle en M\'edecine, Montreal, Canada}
\email{malik.amir@umontreal.ca}

\author{Sourangshu Ghosh}
\thanks{\textbf{All authors contributed equally to this work.}}
\address{Department of Civil Engineering, Indian Institute of Science Bangalore, India}
\email{sourangshug@iisc.ac.in}

\subjclass[2020]{Primary 62F10; Secondary 14P25, 53B20, 53C21, 62B10}
\keywords{information geometry, higher-order asymptotics, statistical manifolds, Fisher--Rao metric, statistical curvature, asymptotic inference, point estimation, singular statistical models, algebraic geometry}

\begin{document}

\begin{abstract}
Classical asymptotics based on Fisher information describe the covariance of regular efficient estimators through a local quadratic approximation to the log likelihood. They therefore reflect only the geometry of the statistical model to first order. In curved settings, including mixtures, curved exponential families, models with latent variables, and models with parameters constrained to a manifold, the behavior at finite sample sizes can deviate systematically from these predictions.

We develop a refinement of classical covariance asymptotics that incorporates curvature and is invariant under changes of coordinates. We equip a regular parametric family \((\Theta,g)\) with the Fisher information metric and immerse it in the Hilbert space \(L^2(\mu)\) by taking the square root of each density. Combining higher order likelihood expansions with the intrinsic and extrinsic geometry of this immersion, we derive a covariance correction of order \(1/n^2\) for estimators that solve the score equation and are efficient to first order. The leading term is the inverse Fisher information divided by \(n\). The expansion holds under suitable regularity, moment, and replacement assumptions.

The correction is determined by a tensor \(P_{ij}\) that decomposes into a multiple of the Ricci tensor, a Gram contraction of the second fundamental form, and a Hellinger discrepancy tensor that accounts for higher order probabilistic information beyond the geometry of the immersion. The Gram contraction is positive semidefinite, and the full correction is invariant under smooth reparameterization. The correction vanishes in full exponential families with a quadratic log partition function. In one dimension the intrinsic term vanishes, and the remaining correction is related to Efron's statistical curvature.

We also consider singular models in which the Fisher information degenerates. Under an additive normal crossing assumption, we use resolution of singularities to compute the real log canonical threshold and determine the rates of the posterior mean squared error and covariance. We give the leading covariance coefficient explicitly in terms of the resolution data and recover the classical leading term in the regular case. The resolved Hessian metric is flat under the additive assumption, while the pullback of the Fisher information metric describes the intrinsic curvature on its regular locus. These results suggest geometric diagnostics for weak identifiability and motivate further work on regularization and optimization in probabilistic models.

We thank Open Math OS for formalizing this article in Lean.
\end{abstract}

\maketitle

\tableofcontents

\clearpage

\section*{Notation}
\label{sec:notation}

The following table collects the notation used throughout the article.

\begin{center}
\footnotesize
\setlength{\tabcolsep}{2pt}
\renewcommand{\arraystretch}{1.10}
\begin{tabular}{@{}
  >{\raggedright\arraybackslash}p{0.19\textwidth}
  >{\raggedright\arraybackslash}p{0.29\textwidth}
  >{\raggedright\arraybackslash}p{0.19\textwidth}
  >{\raggedright\arraybackslash}p{0.29\textwidth}
  @{}}
\toprule
\textbf{Notation} & \textbf{Meaning} &
\textbf{Notation} & \textbf{Meaning} \\
\midrule
$\Theta\subset\mathbb R^d$, $\theta$
  & Parameter space of dimension $d$ and parameter
  & $\Theta_{\mathrm{reg}}$
  & Regular locus $\{\theta\in\Theta\mid I(\theta)\succ0\}$ \\

$p_\theta$, $\mu$
  & Density and dominating measure on $(\mathcal X,\mathcal A)$
  & $\ell(\theta)$
  & Sample log likelihood $\sum_{k=1}^n\log p_\theta(X_k)$ \\

$\psi_\theta=\sqrt{p_\theta}$, $\Psi$
  & Square root of the density and Hellinger map
  & $U_i,U_{ij},U_{ijk}$
  & $\partial_i\ell$, $\partial_i\partial_j\ell$,
    $\partial_i\partial_j\partial_k\ell$ \\

$L^2(\mu)$, $\langle\cdot,\cdot\rangle_{L^2}$
  & Ambient Hilbert space and its inner product
  & $s_i,s_{ij},s_{ijk}$
  & Corresponding derivatives of $\log p_\theta(x)$ for one observation \\

$\partial_i$
  & Coordinate derivative $\partial/\partial\theta^i$
  & $T_{ijk}$
  & Cubic score moment $\E_\theta[s_is_js_k]$ \\

$g_{ij}(\theta)$
  & Fisher information metric\newline
    $4\langle\partial_i\psi_\theta,\partial_j\psi_\theta\rangle_{L^2}$
  & $\Gamma^{(e)}_{ijk}$
  & Coefficients of the exponential connection\newline
    $\E_\theta[s_{ij}s_k]$ \\

$I(\theta)=(I_{ij})$, $g^{ij}$
  & Fisher information $I=g$ and entries of the inverse metric
  & $\kappa_{ijk}$
  & Expected third derivative $\E_\theta[s_{ijk}]$ \\

$\Gamma^k_{ij}$
  & Christoffel symbols of $g$
  & $\sigma_{ikjm}$
  & Mixed moment $\E_\theta[s_{ik}s_{jm}]$ \\

$\nabla$, $(\nabla^2\ell)_{ij}$
  & Levi Civita connection of $g$ and covariant Hessian of $\ell$
  & $\E_\theta$, $\Cov_\theta$, $\mathrm{Bias}_\theta$
  & Expectation, covariance, and bias under $p_\theta$ \\

$R^m{}_{jkl}$, $R_{ijkl}$
  & Riemann curvature tensor in its $(1,3)$ and $(0,4)$ forms
  & $\hat\theta_n$, $\Delta_n$
  & Estimator and estimation error $\Delta_n=\hat\theta_n-\theta$ \\

$\mathrm{Ric}_{ij}$
  & Ricci tensor $R^m{}_{imj}$
  & $A_n,B_n,R_n$
  & Terms in the expansion
    $\Delta_n=A_n/\sqrt n+B_n/n+R_n$ \\

$\mathrm{II}_{ij}$
  & Second fundamental form
    $\Pi^\perp(\partial_i\partial_j\psi_\theta)$
  & $o_p$, $O_p$
  & Stochastic orders in probability \\

$\mathrm{II}^{\mathrm S}_{ij}$
  & Spherical second fundamental form
    $\mathrm{II}_{ij}+\tfrac14g_{ij}\psi_\theta$
  & $\xrightarrow{d}$
  & Convergence in distribution \\

$\kappa^2$
  & Squared norm of the second fundamental form\newline
    $g^{ik}g^{j\ell}\langle\mathrm{II}_{ij},\mathrm{II}_{k\ell}\rangle_{L^2}$
  & $K(\theta)$, $K(\pi(u))$
  & Kullback Leibler divergence and its pullback \\

$R^\sharp_{ij}$
  & Intrinsic contraction $g^{k\ell}R_{ikj\ell}=\mathrm{Ric}_{ij}$
  & $\lambda$
  & Real log canonical threshold \\

$S^\sharp_{ij}$
  & Gram contraction\newline
    $g^{k\ell}\langle\mathrm{II}_{ik},\mathrm{II}_{j\ell}\rangle_{L^2}$
  & $\zeta(s)$
  & Zeta function defining the real log canonical threshold \\

$H_{ij}$
  & Contraction with the trace of $\mathrm{II}$\newline
    $g^{k\ell}\langle\mathrm{II}_{ij},\mathrm{II}_{k\ell}\rangle_{L^2}$
  & $\pi$, $u$
  & Resolution map and resolved coordinates \\

$P_{ij}$, $D_{ij}$
  & Correction tensor and Hellinger discrepancy tensor
  & $c_j$, $k_j$, $h_j$
  & Coefficients and exponents in $\sum_jc_ju_j^{2k_j}$,
    and Jacobian exponents \\

$C(\theta)=I^{-1}PI^{-1}$
  & Coefficient of $1/n^2$ in the covariance expansion
  & $\mathcal T_{\theta_0}$
  & Zero set of the leading homogeneous term of $K$ at $\theta_0$ \\

$\|v\|_g^2$
  & Squared norm $g_{ij}v^iv^j$
  & $e_a$
  & Vectors of an orthonormal basis of $T_\theta\Theta$ for $g$ \\

$\gamma^2(\theta)$
  & Efron's statistical curvature in dimension one
  & $n$
  & Sample size \\
\bottomrule
\end{tabular}
\end{center}

\section{Introduction}\label{introduction}

Let \(X_1,\ldots,X_n\) be independent observations with common density
\(p_\theta\) in a smooth parametric family
\(\{p_\theta\}_{\theta\in\Theta}\), where \(\Theta\) is an open subset of
\(\mathbb R^d\). Under the regularity assumptions of
Section~\ref{stat_manifold}, the score for one observation has mean zero
and its covariance defines the Fisher information,
\[
s_\theta(X)=\nabla_\theta\log p_\theta(X),
\qquad
I(\theta)=\E_\theta\!\left[s_\theta(X)s_\theta(X)^\top\right].
\]
For a regular estimator \(\hat\theta_n\) that is asymptotically
efficient at \(\theta\), the local quadratic approximation to the log
likelihood yields the limiting distribution
\[
\sqrt n\,(\hat\theta_n-\theta)
\xrightarrow{d}
\mathcal N\!\bigl(0,I(\theta)^{-1}\bigr).
\]
If the sequence \(n\|\hat\theta_n-\theta\|^2\) is also uniformly
integrable under \(p_\theta\), this convergence yields the covariance
expansion
\[
\Cov_\theta(\hat\theta_n)
=
\frac{1}{n}I(\theta)^{-1}
+
o\!\left(\frac{1}{n}\right).
\]
The leading term depends only on the Fisher information at the true
parameter, leaving open how higher moments of the likelihood
derivatives enter the coefficient at order \(1/n^2\) and how that
coefficient depends on the coordinates used to express the estimation
error.

To study this coefficient, we equip the regular locus with the
Riemannian metric \(g=I\) and consider the Hellinger immersion
\[
\theta \longmapsto \Psi(\theta)=\sqrt{p_\theta}\in L^2(\mu),
\qquad \theta\in\Theta_{\mathrm{reg}}.
\]
The metric \(g\) is four times the pullback of the ambient inner
product, so the immersion gives rise to both the intrinsic curvature
of \(g\) and the second fundamental form of \(\Psi\), which records the
normal component of its bending in \(L^2(\mu)\). The Gauss equation
relates the two by expressing the Riemann curvature tensor in terms
of the second fundamental form.

For estimators that solve the score equation, Theorem~\ref{thm:main}
identifies the coefficient at order \(1/n^2\) in Fisher normal
coordinates centered at \(\theta\), assuming the covariance has an
expansion to this order and the stated moment and replacement
conditions hold. The coefficient is \(I^{-1}PI^{-1}\), where
\[
P=\tfrac12 R^\sharp+S^\sharp+D.
\]
Here \(R^\sharp\) is the Ricci tensor, \(S^\sharp\) is a positive
semidefinite Gram contraction of the second fundamental form, and
\(D\), the Hellinger discrepancy, is the residual after subtracting
the two geometric terms from the tensor obtained by the likelihood
calculation. In dimension one the Ricci term vanishes, and in a full
exponential family with a quadratic log partition function the
entire correction vanishes under the hypotheses of the theorem.

When the Fisher information degenerates at the true parameter,
we study the measure whose density is proportional to
\(\exp(-nK)\) times the prior, where \(K\) is the Kullback Leibler
divergence from the model at the true parameter. Under an additive
normal crossing assumption on the resolved divergence, we compute
the real log canonical threshold and obtain the leading covariance
coefficient from the normal crossing data and the differential of
the resolution map. If the resolved divergence is quadratic and
the resolution map is a local diffeomorphism, the formula recovers
the inverse Fisher information as the coefficient of \(1/n\).

\subsection{Relation to previous work}

The information inequality of Cram\'er and Rao bounds the covariance
of a locally unbiased estimator in terms of the inverse Fisher
information, and its extension to estimators with a general mean map
involves the derivative of that map
\cite{Cramer1999,Kay1993,LehmannCasella1998,Rao1945,BickelKlaassen1993,VanDerVaart2000}.
The interpretation of Fisher information as a Riemannian metric
arises both from the geometry of square root densities and from
invariance principles for sufficient statistics
\cite{Amari2016,AmariNagaoka2000,AyJost2017a,AyJost2017b,Campbell1986,Chentsov1982}.

Inequalities of Bhattacharyya type and higher order likelihood
expansions incorporate derivatives beyond the score, providing a
framework for studying the accuracy of the leading Fisher
approximation \cite{Bhattacharyya1946,LehmannCasella1998,BarndorffNielsenCox1994,KassVos2011}.
For a curved family with one parameter, Efron used statistical
curvature to quantify the information lost when the sample is reduced
to the maximum likelihood estimator, and related work expresses
higher order efficiency losses in curved exponential families in
geometric terms \cite{Efron1975,Eguchi1983,KassVos2011}. In nonlinear
regression, Bates and Watts distinguish the bending of the regression
surface from the effects of its parametrization, a distinction
that also motivates our use of geometric quantities defined by the
model. The object studied here is the covariance of the estimator
itself, whose correction we express using the Fisher metric and
the second fundamental form of the Hellinger immersion
\cite{BatesWatts1980,AmariNagaoka2000,AyJost2017b,Lauritzen1987}.

The restriction to the locus where the Fisher information is positive
definite matters in models such as finite mixtures, where the metric
can have nontrivial local geometry on the regular locus and degenerate
elsewhere \cite{KassVos2011,Marriott2002}. At singular points we use
the method of resolution of singularities from singular learning
theory, with the additive hypothesis allowing the integrals for
the covariance coefficient to be evaluated in terms of the Gamma
function \cite{Watanabe2009}.

\subsection{Organization of the article}

Sections~\ref{stat_manifold} and \ref{curv_manifold} develop the Fisher
information metric and the geometry of the Hellinger immersion, and
Section~\ref{second_order_geometry} introduces the covariant likelihood
derivatives and score moments used in the covariance calculation of
Section~\ref{curvature_corrected_CRLB}. Section~\ref{sec:singular_models}
treats the additive singular case, and Section~\ref{sec:conclusion}
discusses further questions concerning probabilistic learning models.
Section~\ref{sec:formal} records the results formalized in Lean~4,
with the standing assumptions and detailed moment calculations for
the main theorem given in Appendix~\ref{appendix}.

\section{Statistical Manifold and Fisher Information Geometry}
\label{stat_manifold}

\subsection{The square root density map}

Let \((\mathcal X,\mathcal A,\mu)\) be a measure space, with
\(\mathcal X\) a countable union of measurable sets of finite measure,
and let \(\{p_\theta\}_{\theta\in\Theta}\) be a family of probability
densities with respect to \(\mu\), where \(\Theta\subset\mathbb R^d\)
is open. For almost every \(x\) with respect to \(\mu\), assume that
\(\theta\mapsto p_\theta(x)\) is positive and of class \(C^3\) on
\(\Theta\), and that for every multiindex \(\alpha\) with
\(|\alpha|\leq 3\), the derivative
\(\partial_\theta^\alpha\sqrt{p_\theta}\) belongs to \(L^2(\mu)\) and
depends continuously on \(\theta\) in the \(L^2(\mu)\) norm.

We also assume that for each \(\theta_0\in\Theta\) there are a
neighborhood \(U\) of \(\theta_0\) and nonnegative measurable functions
\(G_0,G_1,G_2\) with \(G_1,G_2\in L^1(\mu)\) and
\(p_{\theta_0}G_0\in L^1(\mu)\), such that for almost every \(x\)
and every \(\theta\in U\),
\[
\|\partial_\theta p_\theta(x)\|\leq G_1(x),
\qquad
\|\partial_\theta^2p_\theta(x)\|\leq G_2(x),
\qquad
\|\partial_\theta^2\log p_\theta(x)\|\leq G_0(x).
\]
The bounds on the density derivatives justify differentiation under
the integral sign through order two, and the bound on the log density
provides the domination used in the quadratic expansion of the
Kullback Leibler divergence in Proposition~\ref{prop:KL_regular}.

Define the Hellinger map by
\[
\Theta\ni\theta\longmapsto\Psi(\theta)=\psi_\theta=\sqrt{p_\theta}
\in L^2(\mu).
\]
The pointwise fundamental theorem of calculus, applied along coordinate
lines, identifies the \(L^2(\mu)\) derivatives of \(\Psi\) through
order three with the corresponding pointwise derivatives, since
those derivatives are continuous with values in \(L^2(\mu)\).
Thus \(\Psi\) is of class \(C^3\), and its image lies in the unit
sphere
\[
\mathbb S=\{f\in L^2(\mu)\mid\|f\|_{L^2}=1\},
\]
since \(\|\psi_\theta\|_{L^2}^2=\int_{\mathcal X}p_\theta\,d\mu=1\).
The distance between two such images in \(L^2(\mu)\) is proportional
to the Hellinger distance between the corresponding distributions.

\subsection{Tangent space and orthogonality}

Writing \(\partial_i=\partial/\partial\theta^i\), the chain rule
gives, for almost every \(x\),
\[
\partial_i\psi_\theta(x)
=
\frac{\partial_i p_\theta(x)}{2\sqrt{p_\theta(x)}}
=
\frac12\psi_\theta(x)\,\partial_i\log p_\theta(x).
\]
The image of the differential \(d\Psi_\theta\) is therefore spanned by
the vectors \(\partial_i\psi_\theta\), each of which is orthogonal to
\(\psi_\theta\) by differentiation of the identity
\(\langle\psi_\theta,\psi_\theta\rangle_{L^2}=1\). In particular,
\[
d\Psi_\theta(T_\theta\Theta)
=
\operatorname{span}\{\partial_1\psi_\theta,\ldots,\partial_d\psi_\theta\}
\subseteq
T_{\psi_\theta}\mathbb S
=
\{v\in L^2(\mu)\mid\langle v,\psi_\theta\rangle_{L^2}=0\},
\]
where \(\langle f,h\rangle_{L^2}=\int_{\mathcal X}fh\,d\mu\) is
the ambient Hilbert space inner product.

\subsection{The induced Fisher information metric}

We normalize the pullback of the ambient inner product by a factor
of four, so that the chain rule identity gives
\begin{align*}
g_{ij}(\theta)
&=4\langle\partial_i\psi_\theta,\partial_j\psi_\theta\rangle_{L^2}\\
&=\int_{\mathcal X}
  \partial_i\log p_\theta(x)\,
  \partial_j\log p_\theta(x)\,p_\theta(x)\,d\mu(x)\\
&=\E_\theta\!\left[
  \partial_i\log p_\theta(X)\,\partial_j\log p_\theta(X)
  \right]
=I_{ij}(\theta).
\end{align*}
Since \(I(\theta)\) is four times the Gram matrix of the vectors
\(\partial_i\psi_\theta\), it is positive definite precisely when
\(d\Psi_\theta\) is injective. Consequently, on the regular locus
\[
\Theta_{\mathrm{reg}}
=
\{\theta\in\Theta\mid I(\theta)\succ0\},
\]
the map \(\Psi\) is an immersion of dimension \(d\), and \(g\) is a
Riemannian metric of class \(C^2\). The immersion is a local embedding
near every regular parameter, so the tangent space of the local image
through \(\psi_\theta\) is identified with
\(d\Psi_\theta(T_\theta\Theta)\).

The pullback construction also gives a direct interpretation of the
metric along a smooth curve \(t\mapsto\theta(t)\) through
\(\theta(0)=\theta\), since its squared speed is the expected square
of the log likelihood derivative in the direction of the curve,
\[
\|\dot\theta(0)\|_g^2
=
\E_\theta\!\left[
  \left(
    \left.\frac{d}{dt}\log p_{\theta(t)}(X)\right|_{t=0}
  \right)^2
\right].
\]

\subsection{Geometric significance}

The information inequality of Cram\'er and Rao uses the metric at
the true parameter, and the covariance expansion studied in
Section~\ref{curvature_corrected_CRLB} also involves derivatives of
the metric and of the immersion. The derivatives of the metric
determine its intrinsic curvature, and the normal components of
\(\partial_i\partial_j\psi_\theta\) define the second fundamental form
of \(\Psi\), which Section~\ref{curv_manifold} relates to the Riemann
curvature tensor through the Gauss equation.

\section{Curvature of the Statistical Manifold}\label{curv_manifold}

\subsection{Intrinsic curvature}

On the regular locus \(\Theta_{\mathrm{reg}}\), the Fisher information
defines a Riemannian metric \(g_{ij}(\theta)=I_{ij}(\theta)\) of class
\(C^2\), whose Levi Civita connection is the unique connection that
preserves \(g\) and has zero torsion. In local coordinates
\(\theta=(\theta^1,\ldots,\theta^d)\), its Christoffel symbols are
\[
\Gamma^k_{ij}
=
\frac12 g^{k\ell}
\bigl(
\partial_i g_{j\ell}
+
\partial_j g_{i\ell}
-
\partial_\ell g_{ij}
\bigr),
\]
where \(g^{k\ell}\) are the entries of the inverse metric and repeated
upper and lower indices are summed. We use the curvature convention
\begin{align*}
R^m{}_{jk\ell}
&=
\partial_k\Gamma^m_{j\ell}
-
\partial_\ell\Gamma^m_{jk}
+
\Gamma^m_{kr}\Gamma^r_{j\ell}
-
\Gamma^m_{\ell r}\Gamma^r_{jk},\\
R_{ijk\ell}
&=
g_{im}R^m{}_{jk\ell}.
\end{align*}

In geodesic normal coordinates \(x\) centered at \(\theta_0\), defined
using an orthonormal basis for \(g(\theta_0)\), the metric has the
expansion
\[
g_{ij}(x)
=
\delta_{ij}
-
\frac13 R_{ikj\ell}(\theta_0)x^kx^\ell
+
o(\|x\|^2).
\]
The curvature tensor thus determines the quadratic term in the normal
coordinate expansion of the metric, and nonzero curvature at
\(\theta_0\) obstructs the existence of Euclidean coordinates on any
neighborhood of that point. If the curvature vanishes throughout an
open set, each point of that set has a neighborhood on which the
metric is Euclidean in suitable coordinates.

\subsection{Extrinsic curvature of the Hellinger immersion}

Let \(\mathcal M\) denote the local image of the Hellinger immersion
near a regular parameter \(\theta\), and let \(\Pi^\perp_\theta\) be
orthogonal projection in \(L^2(\mu)\) onto the complement of
\(d\Psi_\theta(T_\theta\Theta)\). Since the ambient connection is
ordinary differentiation and \(g\) is a constant multiple of the
induced metric, the Gauss formula gives the orthogonal decomposition
\[
\partial_i\partial_j\psi_\theta
=
\Gamma^k_{ij}\partial_k\psi_\theta
+
\mathrm{II}_{ij},
\qquad
\mathrm{II}_{ij}
=
\Pi^\perp_\theta(\partial_i\partial_j\psi_\theta),
\]
where \(\mathrm{II}\) is the second fundamental form of the immersion
into \(L^2(\mu)\), and its symmetry in \(i,j\) follows from the
symmetry of the mixed partial derivatives.

Because \(\|\psi_\theta\|_{L^2}=1\), differentiating
\(\langle\partial_j\psi_\theta,\psi_\theta\rangle_{L^2}=0\) gives
\[
\langle\mathrm{II}_{ij},\psi_\theta\rangle_{L^2}
=
-\langle\partial_i\psi_\theta,\partial_j\psi_\theta\rangle_{L^2}
=
-\frac14 g_{ij}.
\]
The second fundamental form of the immersion into the unit sphere is
therefore
\[
\mathrm{II}^{\mathrm S}_{ij}
=
\mathrm{II}_{ij}+\frac14 g_{ij}\psi_\theta,
\]
which is orthogonal to both the radial vector \(\psi_\theta\) and the
tangent space of \(\mathcal M\). The two components give the identity
\[
\langle\mathrm{II}_{ab},\mathrm{II}_{ce}\rangle_{L^2}
=
\frac1{16}g_{ab}g_{ce}
+
\langle\mathrm{II}^{\mathrm S}_{ab},
        \mathrm{II}^{\mathrm S}_{ce}\rangle_{L^2}.
\]

If \(\mathrm{II}^{\mathrm S}\) vanishes on a neighborhood, then the
derivatives of \(\psi_\theta\) and of the tangent frame
\(\partial_1\psi_\theta,\ldots,\partial_d\psi_\theta\) lie in their
joint span, by the Gauss formula and the expression for the radial
component. This span is consequently constant along curves in the
neighborhood, so the local image is an open subset of the intersection
of the unit sphere with a fixed linear subspace of dimension \(d+1\).
Conversely, for an immersion whose local image is open in such an
intersection, the spherical second fundamental form vanishes.

\subsection{The Gauss equation}

The ambient space \(L^2(\mu)\) is flat, and the Gauss equation for the
metric induced by \(\Psi\), together with
\(g=4\Psi^*\langle\cdot,\cdot\rangle_{L^2}\), gives
\[
R_{ijk\ell}
=
4\bigl(
\langle\mathrm{II}_{ik},\mathrm{II}_{j\ell}\rangle_{L^2}
-
\langle\mathrm{II}_{i\ell},\mathrm{II}_{jk}\rangle_{L^2}
\bigr).
\]
The factor of four comes from lowering the first curvature index with
\(g\), since a constant rescaling of the metric preserves the connection
and the curvature tensor of type \((1,3)\). Substituting the radial
decomposition of \(\mathrm{II}\) yields
\begin{align*}
R_{ijk\ell}
&=
\frac14\bigl(g_{ik}g_{j\ell}-g_{i\ell}g_{jk}\bigr)\\
&\quad+
4\bigl(
\langle\mathrm{II}^{\mathrm S}_{ik},
        \mathrm{II}^{\mathrm S}_{j\ell}\rangle_{L^2}
-
\langle\mathrm{II}^{\mathrm S}_{i\ell},
        \mathrm{II}^{\mathrm S}_{jk}\rangle_{L^2}
\bigr),
\end{align*}
which expresses the intrinsic curvature as the sum of the sphere's
curvature contribution in the Fisher normalization and a quadratic
expression in the spherical second fundamental form.

\subsection{The squared norm of the second fundamental form}

Contracting the two covariant indices of the second fundamental form
with the inverse metric gives the squared Hilbert Schmidt norm
\[
\kappa^2(\theta)
=
g^{ik}g^{j\ell}
\langle\mathrm{II}_{ij},\mathrm{II}_{k\ell}\rangle_{L^2}
=
\sum_{a,b=1}^d
\|\mathrm{II}_\theta(e_a,e_b)\|_{L^2}^2,
\]
where \(e_1,\ldots,e_d\) is any orthonormal basis of \(T_\theta\Theta\)
for \(g\). The contraction is invariant under smooth reparameterization,
and the orthogonality of the radial and spherical components gives
\[
\kappa^2(\theta)
=
\frac{d}{16}
+
\sum_{a,b=1}^d
\|\mathrm{II}^{\mathrm S}_\theta(e_a,e_b)\|_{L^2}^2.
\]
Thus the radial component contributes \(d/16\) at every regular
parameter, with equality \(\kappa^2(\theta)=d/16\) precisely when
\(\mathrm{II}^{\mathrm S}_\theta=0\). With the metric given by the
unscaled pullback, the corresponding radial contribution is \(d\).

\subsection{Relation to Efron's statistical curvature}

For a regular curve in an exponential family, Efron's statistical
curvature \(\gamma^2(\theta)\) determines the leading loss of Fisher
information when the full sample is reduced to the maximum likelihood
estimator. Under the hypotheses of Efron's information expansion, the
sampling distribution of \(\hat\theta_n\) has Fisher information
\[
I_{\mathrm{eff}}(\theta)
=
I_{\hat\theta_n}(\theta)
=
n I(\theta)
\left(
1-\frac{\gamma^2(\theta)}{n}
+
o\!\left(\frac1n\right)
\right),
\]
where \(I(\theta)\) is the Fisher information for a single observation
and \(nI(\theta)\) is that of the full sample \cite{Efron1975}.
The scalar \(\gamma^2\) measures information loss under this reduction,
while \(\kappa^2\) measures the bending of the Hellinger immersion and
includes the radial contribution computed above. The covariance of the
estimator itself is the object of Section~\ref{curvature_corrected_CRLB},
where the coefficient at order \(1/n^2\) is obtained from moments of
the likelihood derivatives.

\subsection{Curvature terms in the covariance expansion}

The covariance expansion uses two covariant tensors of order two,
obtained by contracting the Riemann curvature tensor and the second
fundamental form with the inverse metric,
\[
R^\sharp_{ij}=g^{k\ell}R_{ikj\ell},
\qquad
S^\sharp_{ij}
=
g^{k\ell}
\langle\mathrm{II}_{ik},\mathrm{II}_{j\ell}\rangle_{L^2}.
\]
With the curvature convention above, \(R^\sharp\) is the Ricci tensor,
and \(S^\sharp\) is positive semidefinite by its Gram representation.
These tensors supply the geometric terms in the decomposition of
\(P\) in Theorem~\ref{thm:main}, with the Hellinger discrepancy
accounting for the difference between their contribution and the
coefficient obtained from the likelihood calculation.

\section{Second Order Geometry of the Log Likelihood}
\label{second_order_geometry}

\subsection{Score, Hessian, and Fisher information}

For independent observations \(X_1,\ldots,X_n\) with common density
\(p_\theta\), write \(\ell(\theta)=\sum_{r=1}^n\log p_\theta(X_r)\)
for the sample log likelihood. For one observation, denote the first
three coordinate derivatives of the log density by
\[
s_i(x;\theta)=\partial_i\log p_\theta(x),\qquad
s_{ij}(x;\theta)=\partial_i\partial_j\log p_\theta(x),\qquad
s_{ijk}(x;\theta)=\partial_i\partial_j\partial_k\log p_\theta(x),
\]
and write \(U_i,U_{ij},U_{ijk}\) for the corresponding derivatives
of \(\ell\), obtained by summing over the sample. Below, the argument
\((X;\theta)\) is suppressed inside expectations for one observation.

Differentiating the normalization \(\int_{\mathcal X}p_\theta\,d\mu=1\)
once and twice, under the assumptions of Section~\ref{stat_manifold},
gives \(\E_\theta[s_i]=0\) and
\(\E_\theta[s_{ij}+s_is_j]=0\), so that
\[
\E_\theta[U_i]=0,\qquad
-\E_\theta[U_{ij}]
=
-n\,\E_\theta[s_{ij}]
=
n\,\E_\theta[s_is_j]
=
nI_{ij}(\theta).
\]
Thus the Fisher information determines the mean quadratic term of the
sample log likelihood, whose fluctuations enter the covariance
calculation in Section~\ref{curvature_corrected_CRLB}.

\subsection{Taylor expansion of the log likelihood}

For a fixed sample, Taylor expansion in a coordinate chart gives
\begin{align*}
\ell(\theta+\delta)
&=
\ell(\theta)
+
U_i(\theta)\delta^i
+
\frac12U_{ij}(\theta)\delta^i\delta^j\\
&\quad+
\frac16U_{ijk}(\theta)\delta^i\delta^j\delta^k
+
o(\|\delta\|^3),
\qquad \delta\to0.
\end{align*}
The cubic term is the first possible contribution beyond the
quadratic approximation in the chosen chart, and its behavior under
a change of coordinates involves the lower derivatives of \(\ell\).
The connection of the Fisher information metric provides the
corresponding covariant derivatives.

\subsection{The covariant Hessian}

If \(\eta\) is another coordinate system and \(\theta=\theta(\eta)\),
the chain rule gives
\[
\frac{\partial^2\ell}{\partial\eta^a\partial\eta^b}
=
\frac{\partial\theta^i}{\partial\eta^a}
\frac{\partial\theta^j}{\partial\eta^b}U_{ij}
+
\frac{\partial^2\theta^k}{\partial\eta^a\partial\eta^b}U_k.
\]
The second term is proportional to the score and is canceled by the
connection term in the covariant Hessian, whose components are
\[
(\nabla^2\ell)_{ij}
=
U_{ij}-\Gamma^k_{ij}U_k,
\]
with \(\Gamma^k_{ij}\) the Christoffel symbols of the Fisher information
metric. Since the score has mean zero, the ordinary and covariant
Hessians have the same expectation,
\[
\E_\theta[(\nabla^2\ell)_{ij}]
=
\E_\theta[U_{ij}]
=
-ng_{ij}(\theta),
\]
and the centered random tensor
\((\nabla^2\ell)_{ij}+ng_{ij}\) records the sample fluctuations of the
covariant Hessian.

\subsection{Third order likelihood geometry}

Under the moment assumptions of Appendix~\ref{appendix}, define the
cubic score moment, also called the tensor of Amari and Chentsov,
and the coefficients of the exponential connection by
\[
T_{ijk}=\E_\theta[s_is_js_k],
\qquad
\Gamma^{(e)}_{ijk}=\E_\theta[s_{ij}s_k].
\]
Differentiating \(I_{ij}=\E_\theta[s_is_j]\) gives
\[
\partial_k I_{ij}
=
\Gamma^{(e)}_{ikj}
+
\Gamma^{(e)}_{jki}
+
T_{ijk},
\]
where the cubic term comes from differentiating the density against
which the expectation is taken. The same identity follows by
differentiating the representation
\(I_{ij}=4\langle\partial_i\psi_\theta,\partial_j\psi_\theta\rangle_{L^2}\),
and substituting it into the Christoffel formula gives
\[
g_{ka}\Gamma^a_{ij}
=
\Gamma^{(e)}_{ijk}
+
\frac12T_{ijk}.
\]

The normalization of the density also gives a symmetric formula for
the expected third partial derivative. Indeed, since
\(p_\theta=\psi_\theta^2\) is of class \(C^3\) with values in
\(L^1(\mu)\) by continuity of multiplication from
\(L^2(\mu)\times L^2(\mu)\) to \(L^1(\mu)\), we may differentiate the
normalization three times and obtain
\begin{align*}
\kappa_{ijk}
&\coloneqq \E_\theta[s_{ijk}]\\
&=
-\bigl(
\Gamma^{(e)}_{ijk}
+
\Gamma^{(e)}_{ikj}
+
\Gamma^{(e)}_{jki}
+
T_{ijk}
\bigr)
=
-\partial_k I_{ij}-\Gamma^{(e)}_{ijk},
\end{align*}
so that \(\E_\theta[U_{ijk}]=n\kappa_{ijk}\).

For the covariant third derivative
\((\nabla^3\ell)_{ijk}=\nabla_i(\nabla^2\ell)_{jk}\), the terms
proportional to the score have mean zero, and the connection terms
involving the second derivatives contribute
\[
n\bigl(
g_{ia}\Gamma^a_{jk}
+
g_{ja}\Gamma^a_{ik}
+
g_{ka}\Gamma^a_{ij}
\bigr),
\]
to its expectation. Combining this expression with the formulas above
yields the tensor identity
\[
\E_\theta[(\nabla^3\ell)_{ijk}]
=
\frac n2 T_{ijk}.
\]
At the center of Fisher normal coordinates, where the Christoffel
symbols vanish, the ordinary third derivative has the same expectation,
giving \(\kappa_{ijk}=\tfrac12T_{ijk}\) in the coordinates used for the
covariance calculation.

\section{A Second Order Covariance Expansion}
\label{curvature_corrected_CRLB}

\subsection{Asymptotic setting}

Fix a regular parameter \(\theta\in\Theta_{\mathrm{reg}}\) and express
the estimation error in Fisher normal coordinates centered at
\(\theta\). Equivalently, on a normal neighborhood of \(\theta\) the
error is the tangent vector
\[
\Delta_n=\exp_\theta^{-1}(\hat\theta_n)\in T_\theta\Theta,
\]
where the exponential map is taken for the Fisher information metric.
If the normal neighborhood is local, we use a measurable extension
of this tangent error and impose the moment and remainder assumptions
on the resulting random vector. Its covariance is a tensor on
\(T_\theta\Theta\), represented in an orthonormal basis by the ordinary
covariance matrix of the normal coordinates.

We consider estimators that solve the score equation
\(d\ell(\hat\theta_n)=0\) and admit an expansion of the form
\[
\Delta_n
=
\frac{1}{\sqrt n}A_n+\frac{1}{n}B_n+R_n,
\]
under the conditions of Appendix~\ref{appendix}. For this error, the
matrix of second moments decomposes as
\[
\E_\theta[\Delta_n\Delta_n^\top]
=
\Cov_\theta(\Delta_n)
+
\E_\theta[\Delta_n]\E_\theta[\Delta_n]^\top,
\]
so the bias condition \(\E_\theta[\Delta_n]=o(1/n)\) makes the second
moment and covariance agree through order \(1/n^2\).

\subsection{Geometric contractions}

With the curvature convention of Section~\ref{curv_manifold}, set
\[
R^\sharp_{ij}=g^{k\ell}R_{ikj\ell}=\mathrm{Ric}_{ij},
\qquad
S^\sharp_{ij}
=
g^{k\ell}\langle\mathrm{II}_{ik},\mathrm{II}_{j\ell}\rangle_{L^2}.
\]
The second contraction is positive semidefinite, since for every
\(v\in T_\theta\Theta\) and every orthonormal basis
\(e_1,\ldots,e_d\) for \(g\),
\[
S^\sharp(v,v)
=
\sum_{k=1}^d\|\mathrm{II}_\theta(v,e_k)\|_{L^2}^2.
\]
If \(H_{ij}=g^{k\ell}\langle\mathrm{II}_{ij},\mathrm{II}_{k\ell}\rangle_{L^2}\),
then contraction of the Gauss equation gives
\[
R^\sharp=4(H-S^\sharp),
\qquad
\frac12R^\sharp+S^\sharp
=
2H-S^\sharp
=
\frac14R^\sharp+H.
\]
Thus both terms in the displayed geometric combination are determined
by the second fundamental form.

\subsection{Main theorem}

For one observation, write
\(q_{ij}=s_{ij}-\Gamma^a_{ij}s_a\) for the covariant Hessian of the log
density, so that \(\E_\theta[q_{ij}]=-g_{ij}\), and let
\(T_{ijk}=\E_\theta[s_is_js_k]\) be the cubic score tensor of
Section~\ref{second_order_geometry}.

\begin{theorem}[Geometric decomposition of the covariance correction]
\label{thm:main}
Let the model satisfy the regularity assumptions of
Section~\ref{stat_manifold}, and fix
\(\theta\in\Theta_{\mathrm{reg}}\). Suppose that \(\hat\theta_n\)
solves the score equation and that its error in Fisher normal
coordinates satisfies the conditions in Appendix~\ref{appendix},
items~(1) through~(11) of Section~\ref{ass:standing}. Then
\[
\Cov_\theta(\Delta_n)
=
\frac{1}{n}I(\theta)^{-1}
+
\frac{1}{n^2}I(\theta)^{-1}P(\theta)I(\theta)^{-1}
+
o\!\left(\frac{1}{n^2}\right),
\]
where the symmetric covariant tensor \(P\) is given by
\begin{equation}
\label{eq:P_moments}
\begin{aligned}
P_{ij}
&=
g^{k\ell}\Cov_\theta(q_{ik},q_{j\ell})
-
\frac18g^{kr}g^{\ell s}T_{ik\ell}T_{jrs}\\
&\quad+
\frac1{16}
\bigl(g^{k\ell}T_{ik\ell}\bigr)
\bigl(g^{rs}T_{jrs}\bigr).
\end{aligned}
\end{equation}
Writing \(D\) for the residual after subtracting the geometric terms
gives the decomposition
\begin{equation}
\label{eq:P_decomp}
P=\frac12 R^\sharp+S^\sharp+D.
\end{equation}
The expansion has the following properties.
\begin{enumerate}
\item The Gram contraction \(S^\sharp\) is positive semidefinite.
\item The tensor \(P\) vanishes in a full exponential family with a
      quadratic log partition function, and more generally whenever
      \(T=0\) and
      \(g^{k\ell}\Cov_\theta(q_{ik},q_{j\ell})=0\) for all \(i,j\).
\item In dimension one the Ricci term vanishes, and in orthonormal
      Fisher normal coordinates the coefficient reduces to
      \[
      P_{11}
      =
      \Var_\theta(s_{11})
      -
      \frac14\bigl(\E_\theta[s_{111}]\bigr)^2.
      \]
\end{enumerate}
\end{theorem}

\begin{proof}
The bound \(\E_\theta\|R_n\|^2=o(1/n^3)\), together with the
moment bounds on \(A_n\) and \(B_n\), makes the terms involving
\(R_n\) negligible at order \(1/n^2\) by the Cauchy Schwarz
inequality. The mixed term between \(A_n\) and \(B_n\) is negligible
at the same order by item~(10), and item~(8) gives
\(\E_\theta[A_nA_n^\top]=g^{-1}+o(1/n)\). Combining these estimates
with the bias condition yields
\[
\Cov_\theta(\Delta_n)
=
\frac1n g^{-1}
+
\frac1{n^2}\E_\theta[B_nB_n^\top]
+
o\!\left(\frac1{n^2}\right).
\]

To compute the limit of the second moment of \(B_n\), let
\(H_{ij,n}\) denote the centered covariant Hessian divided by
\(\sqrt n\), and let \((Z,\mathcal H)\) be the joint centered
Gaussian limit of \((A_n,H_n)\) from
Lemma~\ref{lem:joint_clt_wick}. The sampling covariances and the
connection identity of Section~\ref{second_order_geometry} give
\begin{align*}
\E[Z^iZ^j]
&=g^{ij},\\
\E[\mathcal H_{ik}\mathcal H_{j\ell}]
&=\Cov_\theta(q_{ik},q_{j\ell}),\\
\E[\mathcal H_{ij}Z^k]
&=g^{ka}\E_\theta[(q_{ij}+g_{ij})s_a]
=-\frac12g^{ka}T_{ija}.
\end{align*}
At the center of the normal coordinates,
\(\E_\theta[s_{ijk}]=\tfrac12T_{ijk}\), so the replacement in
item~(11) and the convergence of moments in
Lemma~\ref{lem:joint_clt_wick} identify the limit of
\(\E_\theta[B_{i,n}B_{j,n}]\), with \(B_{i,n}=g_{ia}B_n^a\),
as \(\E[\mathcal B_i\mathcal B_j]\), where
\[
\mathcal B_i
=
\mathcal H_{ij}Z^j+\frac14T_{ijk}Z^jZ^k.
\]

Set \(W_{ij}=\mathcal H_{ij}+\tfrac12T_{ijk}Z^k\), which is
independent of \(Z\) because the two are jointly Gaussian and their
cross covariance vanishes. Since
\(\mathcal B_i=W_{ij}Z^j-\tfrac14T_{ijk}Z^jZ^k\), the mixed terms
have mean zero, and the remaining terms give
\[
\E[\mathcal B_i\mathcal B_j]
=
g^{k\ell}\E[W_{ik}W_{j\ell}]
+
\frac1{16}T_{ik\ell}T_{jrs}\E[Z^kZ^\ell Z^rZ^s].
\]
The first term uses
\[
\E[W_{ik}W_{j\ell}]
=
\Cov_\theta(q_{ik},q_{j\ell})
-
\frac14g^{ab}T_{ika}T_{j\ell b},
\]
and the Gaussian fourth moment is
\(g^{k\ell}g^{rs}+g^{kr}g^{\ell s}+g^{ks}g^{\ell r}\).
Substitution yields \eqref{eq:P_moments}, and raising the two
indices gives the limit of \(\E_\theta[B_nB_n^\top]\) in the
covariance expansion. This also gives \(P\succeq0\) under the
stated hypotheses, since \(P\) is the second moment of
\(\mathcal B\).

The Gram representation proves item~(1). For item~(2), in natural
coordinates a full exponential family with a quadratic log partition
function has constant Fisher information, deterministic
\(q_{ij}=-g_{ij}\), and \(T=0\), so every term in
\eqref{eq:P_moments} vanishes. Finally, in dimension one the
curvature tensor vanishes by antisymmetry, and at the center of
orthonormal normal coordinates \(q_{11}=s_{11}\) and
\(T_{111}=2\E_\theta[s_{111}]\), which gives item~(3).
\end{proof}

\begin{remark}[The Hellinger discrepancy]
\label{rem:discrepancy}
The residual \(D\) compares the likelihood expression
\eqref{eq:P_moments} with the contractions of the second fundamental
form. Indeed, the Gauss formula and differentiation of the square root
density give
\[
\mathrm{II}_{ij}
=
\frac12\bigl(q_{ij}+\tfrac12s_is_j\bigr)\psi_\theta,
\]
and consequently
\[
4\langle\mathrm{II}_{ij},\mathrm{II}_{k\ell}\rangle_{L^2}
=
\E_\theta\!\left[
  \bigl(q_{ij}+\tfrac12s_is_j\bigr)
  \bigl(q_{k\ell}+\tfrac12s_ks_\ell\bigr)
\right].
\]
The geometric contractions therefore involve fourth score moments
and mixed moments of the covariant Hessian with products of two
scores, while \eqref{eq:P_moments} uses the covariance of the
covariant Hessian and contractions of the cubic score tensor. In
dimension one the residual remains in the formula \(P=S^\sharp+D\),
and when \(P=0\) it cancels the entire geometric combination
\(\tfrac12R^\sharp+S^\sharp\).
\end{remark}

\subsection{Interpretation of the correction}

The covariance in Theorem~\ref{thm:main} is that of the Fisher
normal error, so its coefficient represents a bilinear form on
the cotangent space at \(\theta\). For the ordinary covariance of
the estimator in a nonlinear chart, the quadratic and cubic terms
of the change of coordinates can also contribute at order \(1/n^2\).

For example, in the normal family with mean \(m\) and variance
\(\sigma^2\), the Fisher metric is
\(\sigma^{-2}(dm^2+2\,d\sigma^2)\), so its Ricci tensor is
\(-\tfrac12g\). In the natural coordinates of this family, the
maximum likelihood estimator of
\(\eta_2=-1/(2\sigma^2)\) is
\(\hat\eta_2=-n/(2\sigma^2Q)\), where
\(Q\) has a chi squared distribution with \(n-1\) degrees of
freedom. For \(n>5\), the inverse moments of \(Q\) give
\[
\Var_\theta(\hat\eta_2)
=
\frac{n^2}{2\sigma^4(n-3)^2(n-5)}
=
\frac{1}{2\sigma^4n}
+
\frac{11}{2\sigma^4n^2}
+
O\!\left(\frac1{n^3}\right),
\]
which is an expansion for the ordinary covariance in that chart.

\subsection{Relation to higher order likelihood asymptotics}

Classical inequalities of Bhattacharyya type and higher order
likelihood expansions use derivatives beyond the score to refine
the Fisher information approximation
\cite{Bhattacharyya1946,BarndorffNielsenCox1994,KassVos2011}.
Here the covariance coefficient is obtained by expanding the
score equation and evaluating the resulting Gaussian moments,
which gives the intrinsic likelihood expression
\eqref{eq:P_moments}. The Gauss equation then relates its geometric
decomposition to the curvature of the Fisher metric and the
second fundamental form of the Hellinger immersion.

\section{Extension to Singular Models}
\label{sec:singular_models}

\subsection{The divergence at a singular parameter}

Fix a parameter \(\theta_0\in\Theta\), and write \(K\) for the
Kullback Leibler divergence from \(p_{\theta_0}\) to \(p_\theta\),
\[
K(\theta)
=
\int_{\mathcal X}
p_{\theta_0}(x)\log\frac{p_{\theta_0}(x)}{p_\theta(x)}\,d\mu(x).
\]
The quadratic term of \(K\) is determined by the Fisher information
at \(\theta_0\), including when that information is singular.

\begin{proposition}[Quadratic approximation of the divergence]
\label{prop:KL_regular}
Under the assumptions of Section~\ref{stat_manifold},
\[
K(\theta_0+h)
=
\frac12 h^\top I(\theta_0)h+o(\|h\|^2)
\qquad\text{as }h\to0.
\]
\end{proposition}

\begin{proof}
The domination assumptions in Section~\ref{stat_manifold} allow
differentiation of \(K\) through order two near \(\theta_0\). At
\(\theta_0\), the score and Hessian identities of
Section~\ref{second_order_geometry} give
\[
DK(\theta_0)=0,
\qquad
D^2K(\theta_0)=I(\theta_0).
\]
Taylor expansion gives the result, since \(K(\theta_0)=0\).
\end{proof}

\begin{proposition}[Null directions of the Fisher information]
\label{prop:null_directions}
If \(\operatorname{rank}I(\theta_0)=r<d\), then
\(\mathcal N=\ker I(\theta_0)\) has dimension \(d-r\), and each
\(v\in\mathcal N\) satisfies
\[
\left.\frac{d}{dt}\log p_{\theta_0+tv}(x)\right|_{t=0}=0,
\]
for almost every \(x\) with respect to \(\mu\). The directional
derivative of \(p_\theta(x)\) also vanishes, and
\(K(\theta_0+tv)=o(t^2)\) as \(t\to0\).
\end{proposition}

\begin{proof}
The dimension follows from the rank theorem, and for
\(v\in\mathcal N\),
\[
0
=
v^\top I(\theta_0)v
=
\int_{\mathcal X}
\left(
  \left.\frac{d}{dt}\log p_{\theta_0+tv}(x)\right|_{t=0}
\right)^2
p_{\theta_0}(x)\,d\mu(x).
\]
Since \(p_{\theta_0}\) is positive almost everywhere, the directional
score vanishes almost everywhere, and multiplying it by
\(p_{\theta_0}\) gives the assertion for the density. The assertion
for \(K\) follows from Proposition~\ref{prop:KL_regular}.
\end{proof}

This degeneracy is a statement about the differential of the
parametrization, and a model can still be identifiable near such a
point. For example, the family \(p_t=N(t^3,1)\) is identifiable,
with \(I(0)=0\) and \(K(t)=t^6/2\) at the true parameter \(t=0\).

\subsection{Resolution and the additive normal form}

When \(K\) is real analytic, resolution of singularities gives local
charts in which \(K\) pulls back to a monomial times a positive
analytic factor. In these charts the pullback has the form
\[
K(\pi(u))=\varphi_K(u)\prod_{j\in J_0}u_j^{2k_j},
\qquad
\varphi_K(u)>0,
\]
where \(\pi\) is the resolution map and \(J_0\) indexes the
coordinate divisors in the chart \cite{Watanabe2009}.

\begin{remark}[The additive assumption]
\label{rem:normal_crossing}
For the explicit moment calculations, we assume that on a bounded
coordinate box \(U\subset\mathbb R^d\) containing \(0\), a real
analytic map \(\pi\) with \(\pi(0)=\theta_0\) is injective off a set
of Lebesgue measure zero and satisfies
\begin{equation}
\label{eq:additive_normal_form}
F(u)\coloneqq K(\pi(u))=\sum_{j=1}^d c_j u_j^{2k_j},
\qquad
c_j>0,\qquad k_j\in\mathbb N_{\geq1}.
\end{equation}
This additive form is a separate hypothesis, and every coordinate
appears in the sum, so the resolved zero is isolated in \(U\).
In particular, the hypothesis excludes charts in which the fiber
over \(\theta_0\) has positive dimension through \(0\). The family
\(p_t=N(t^3,1)\) above satisfies the additive hypothesis with the
identity map.
\end{remark}

\begin{proposition}[Jacobian of a blowup chart]
\label{prop:jacobian_resolution}
In the coordinate chart of a blowup given by
\(x_1=u_1\), \(x_j=u_1u_j\) for \(2\leq j\leq s\), and
\(x_j=u_j\) for \(j>s\), the Jacobian determinant is \(u_1^{s-1}\).
For a composition whose Jacobian divisor has normal crossings, the
absolute Jacobian has the local form
\begin{equation}
\label{eq:resolved_jacobian}
|\det D\pi(u)|
=
\varphi_{\mathrm J}(u)\prod_{j=1}^d |u_j|^{h_j},
\qquad
\varphi_{\mathrm J}(u)>0,\qquad h_j\in\mathbb Z_{\geq0}.
\end{equation}
\end{proposition}

\begin{proof}
The Jacobian matrix of the displayed chart is lower triangular,
with \(u_1\) appearing \(s-1\) times on the diagonal and all remaining
diagonal entries equal to one. For a composition, the chain rule
multiplies the Jacobians, and the normal crossing hypothesis expresses
the resulting divisor in coordinate hyperplanes, with multiplicities
\(h_j\) and a nonvanishing factor.
\end{proof}

The calculations below use \eqref{eq:resolved_jacobian} as a
hypothesis on the same chart as \eqref{eq:additive_normal_form}.

\subsection{The Hessian and the pullback information metric}

\begin{proposition}[Hessian in the additive chart]
\label{prop:hessian_resolved}
The coordinate Hessian \(G=D^2F\) is diagonal, with entries
\[
G_{jj}(u)=2k_j(2k_j-1)c_j u_j^{2k_j-2}.
\]
It is positive definite on
\[
U_G
=
\{u\in U\mid u_j\ne0\text{ whenever }k_j>1\},
\]
and its rank at any point is \(d\) minus the number of indices
\(j\) for which \(u_j=0\) and \(k_j>1\).
\end{proposition}

\begin{proof}
Each summand of \(F\) depends on a single coordinate, so all mixed
derivatives vanish and differentiating the individual summands gives
the displayed diagonal entries.
\end{proof}

\begin{proposition}[Connection and curvature of the Hessian metric]
\label{prop:curvature_resolved}
On \(U_G\), the Christoffel symbols of \(G\) are
\[
\widetilde\Gamma^m_{ij}
=
\begin{cases}
(k_m-1)/u_m, & m=i=j\text{ and }k_m>1,\\
0, & \text{otherwise},
\end{cases}
\]
and its Riemann curvature tensor vanishes.
\end{proposition}

\begin{proof}
Each diagonal entry \(G_{jj}\) depends only on \(u_j\), so the
Christoffel formula gives
\(\widetilde\Gamma^j_{jj}=\tfrac12G_{jj}^{-1}\partial_jG_{jj}\),
with all mixed symbols equal to zero. On \(U_G\), the local
changes of coordinate
\[
dy_j=\sqrt{G_{jj}(u_j)}\,du_j,
\]
give \(G=\sum_jdy_j^2\), which proves flatness.
\end{proof}

\begin{remark}[Flatness under additivity]
\label{rem:resolved_flat}
The flatness in Proposition~\ref{prop:curvature_resolved} follows
from the product structure of the additive Hessian. For a
multiplicative form, the coordinate Hessian can be indefinite, as
in the example \(F(u)=u_1^2u_2^2\), whose Hessian determinant is
\(-12u_1^2u_2^2\).
\end{remark}

\begin{remark}[The metric induced by the Hellinger map]
\label{rem:two_metrics}
The map \(\Phi(u)=\Psi(\pi(u))\) induces the form
\[
\hat g_{ij}(u)
=
\langle\partial_i\Phi,\partial_j\Phi\rangle_{L^2}
=
\frac14\partial_i\pi^a\,I_{ab}(\pi(u))\,\partial_j\pi^b.
\]
At \(u=0\), the chain rule and \(DK(\theta_0)=0\) give
\[
G(0)=D\pi(0)^\top I(\theta_0)D\pi(0)=4\hat g(0),
\]
so both forms have rank equal to the number of indices with
\(k_j=1\) there. Away from the minimum, the Hessian also contains
terms involving \(DK\) and the second derivatives of \(\pi\).
\end{remark}

On the locus where \(\hat g\) is positive definite, \(\Phi\) is an
isometric immersion for \(\hat g\). Its second fundamental form
\(\widetilde{\mathrm{II}}_{ij}=\Pi^\perp(\partial_i\partial_j\Phi)\)
therefore satisfies the Gauss and Codazzi equations,
\begin{align*}
\hat R_{ijk\ell}
&=
\langle\widetilde{\mathrm{II}}_{ik},
        \widetilde{\mathrm{II}}_{j\ell}\rangle_{L^2}
-
\langle\widetilde{\mathrm{II}}_{i\ell},
        \widetilde{\mathrm{II}}_{jk}\rangle_{L^2},\\
(\nabla^\perp_i\widetilde{\mathrm{II}})_{jk}
&=
(\nabla^\perp_j\widetilde{\mathrm{II}})_{ik},
\end{align*}
where \(\nabla^\perp\) uses the Levi Civita connection of \(\hat g\)
on the tangent indices and the normal connection on the values of
\(\widetilde{\mathrm{II}}\). Wherever \(\pi\) is a local
diffeomorphism over the regular locus of the original model,
\(\hat g=\tfrac14\pi^*g\), so its curvature there is the pullback
of the original curvature with the corresponding normalization.

\subsection{The real log canonical threshold}

\begin{definition}[Real log canonical threshold]
\label{def:RLCT}
For a nonnegative function \(K\) and a nonnegative density
\(\varphi\), the local real log canonical threshold at the zero
under consideration is the supremum of the real numbers \(s\)
for which \(K^{-s}\) is locally integrable against
\(\varphi(\theta)\,d\theta\). In the analytic setting with a smooth
cutoff, it is equivalently the absolute value of the rightmost
pole of the meromorphic continuation of the local zeta integral
\[
\zeta(z)=\int K(\theta)^z\varphi(\theta)\,d\theta,
\]
\cite{Watanabe2009}.
\end{definition}

Let \(\varphi\) be a prior density, and assume that the weight
\[
\rho(u)=\varphi(\pi(u))\varphi_{\mathrm J}(u),
\]
is bounded on \(U\), continuous at \(0\), and positive there.
We write \(Z_n\) for the local Laplace integral over \(\pi(U)\),
and \(\pi_n\) for the probability measure on \(\pi(U)\) with
density \(Z_n^{-1}e^{-nK(\theta)}\varphi(\theta)\). Thus \(\pi_n\)
is the Gibbs measure formed from the population divergence at
\(\theta_0\).

\begin{proposition}[Laplace asymptotics in the additive chart]
\label{prop:RLCT_additive}
Under \eqref{eq:additive_normal_form} and
\eqref{eq:resolved_jacobian}, the local threshold and normalizing
constant are given by
\[
\lambda=\sum_{j=1}^d\frac{h_j+1}{2k_j},
\qquad
Z_n\sim
\rho(0)\,n^{-\lambda}\prod_{j=1}^d M_j(h_j),
\]
where
\begin{equation}
\label{eq:Mj}
M_j(m)
=
\int_{\mathbb R}e^{-c_jt^{2k_j}}|t|^m\,dt
=
\frac1{k_j}c_j^{-(m+1)/(2k_j)}
\Gamma\!\left(\frac{m+1}{2k_j}\right),
\qquad m\geq0.
\end{equation}
\end{proposition}

\begin{proof}
The change of variables \(\theta=\pi(u)\) gives
\[
Z_n
=
\int_U e^{-nF(u)}\rho(u)\prod_{j=1}^d|u_j|^{h_j}\,du.
\]
Under \(u_j=n^{-1/(2k_j)}t_j\), the Jacobian and monomial weight
together contribute \(n^{-\lambda}\). The rescaled domains exhaust
\(\mathbb R^d\), and dominated convergence gives the stated product
of integrals, which evaluate by the substitution
\(y=c_jt^{2k_j}\) in each half line.

For the threshold, the corresponding dilation
\(u_j=x^{1/(2k_j)}t_j\) shows that the volume of
\(\{F\leq x\}\) against \(\prod_j|u_j|^{h_j}\,du\) is a positive
constant times \(x^\lambda\) for small \(x>0\). Including \(\rho\)
gives a volume asymptotic to a positive multiple of \(x^\lambda\),
so integration against the resulting sublevel volume gives local
integrability of \(F^{-s}\) exactly when \(s<\lambda\).
\end{proof}

\begin{remark}[The multiplicative chart]
\label{rem:watanabe_comparison}
In a multiplicative chart with positive amplitude at the
intersection of the relevant coordinate divisors, the local
integrability conditions are \(h_j-2k_js>-1\) for all
\(j\in J_0\), which give the threshold
\(\min_{j\in J_0}(h_j+1)/(2k_j)\). If the minimum is attained by
\(m\) divisors meeting in the chart, the corresponding Laplace
asymptotics carry a factor \((\log n)^{m-1}\)
\cite{Watanabe2009}.
\end{remark}

\subsection{The zero set of the leading form}

\begin{definition}[Cone defined by the leading form]
Let \(K\) be real analytic and nonconstant near a minimum
\(\theta_0\) with \(K(\theta_0)=0\), and let \(\Phi_{2k}\) be its
first nonzero homogeneous Taylor term, whose degree is even by
nonnegativity. We write
\[
\mathcal T_{\theta_0}
=
\{v\in\mathbb R^d\mid\Phi_{2k}(v)=0\}.
\]
\end{definition}

\begin{proposition}[Directions of higher order vanishing]
The set \(\mathcal T_{\theta_0}\) is a closed cone, and
\[
K(\theta_0+tv)
=
t^{2k}\Phi_{2k}(v)+o(t^{2k})
\qquad\text{as }t\to0.
\]
The tangent cone of the zero set \(\{K=0\}\) is contained in
\(\mathcal T_{\theta_0}\).
\end{proposition}

\begin{proof}
Closedness and the cone property follow from continuity and
homogeneity of \(\Phi_{2k}\), and the expansion is Taylor's
formula. If a sequence of zeros approaches \(\theta_0\) along
directions converging to \(v\), division by the corresponding
distances to the power \(2k\) gives \(\Phi_{2k}(v)=0\).
\end{proof}

For example, \(K(x,y)=x^2+y^4\) gives the line \(x=0\) as the
zero set of the leading form, and the zero set of \(K\) itself
has tangent cone \(\{0\}\). The polarization
\[
\mathcal G(v,w)
=
\frac12\bigl(\Phi_{2k}(v+w)-\Phi_{2k}(v)-\Phi_{2k}(w)\bigr),
\]
is the bilinear form \(\tfrac12v^\top I(\theta_0)w\) when \(k=1\).
For \(k>1\), its joint homogeneity has degree \(2k\), so it
describes the leading form without furnishing a bilinear metric.

\begin{remark}[The leading form in the additive chart]
\label{rem:cone_additive}
For \(F=K\circ\pi\), the leading ordinary homogeneous term is
\(\sum_{j:k_j=k_{\min}}c_jv_j^{2k_{\min}}\), where
\(k_{\min}=\min_j k_j\), so its zero cone in the resolved
coordinates is given by \(v_j=0\) for every \(j\) with
\(k_j=k_{\min}\). The largest exponent \(k_{\max}\) determines
the slowest concentration scale in the resolved coordinates,
and its effect on the original parameter is given by
Theorem~\ref{prop:singular_covariance}.
\end{remark}

\subsection{Mean and covariance of the Gibbs measure}

Let \(\vartheta\) denote the parameter under \(\pi_n\), and put
\[
\beta=\min_j\frac1{k_j},
\qquad
J=\{j\mid 1/k_j=\beta\},
\qquad
a_j=\partial_j\pi(0),
\qquad
v_j=\frac{M_j(h_j+2)}{M_j(h_j)}.
\]

\begin{theorem}[Covariance under the additive normal form]
\label{prop:singular_covariance}
Under the assumptions of Proposition~\ref{prop:RLCT_additive},
\begin{equation}
\label{eq:cov_sing}
\Cov_{\pi_n}(\vartheta)
=
n^{-\beta}\sum_{j\in J}v_j a_ja_j^\top+o(n^{-\beta}),
\qquad
v_j
=
c_j^{-1/k_j}
\frac{\Gamma\!\left(\frac{h_j+3}{2k_j}\right)}
     {\Gamma\!\left(\frac{h_j+1}{2k_j}\right)}.
\end{equation}
The displayed leading coefficient is nonzero if at least one
\(a_j\) with \(j\in J\) is nonzero. If \(\rho\) is also of class
\(C^1\) near \(0\), the Gibbs mean \(m_n=\E_{\pi_n}[\vartheta]\)
satisfies
\begin{equation}
\label{eq:bias}
m_n-\theta_0
=
\sum_{j=1}^d n^{-1/k_j}v_j
\left(
\frac12\partial_j^2\pi(0)
+
a_j\,\partial_j\log\rho(0)
\right)
+
o(n^{-\beta}).
\end{equation}
\end{theorem}

\begin{proof}
Under the law in the resolved chart, set
\(V_{n,j}=n^{1/(2k_j)}u_j\). The rescaling in
Proposition~\ref{prop:RLCT_additive} shows that \(V_n\) converges
in distribution and in every polynomial moment to a vector
\(V\) with independent coordinates, each with density
\[
f_j(t)=M_j(h_j)^{-1}e^{-c_jt^{2k_j}}|t|^{h_j}.
\]
Indeed, the same domination applies after multiplication by
any fixed polynomial. The evenness of the limiting densities
gives
\begin{equation}
\label{eq:moments}
\E[V_j]=0,\qquad
\E[V_j^2]=v_j,\qquad
\E[V_jV_\ell]=0\quad\text{for }j\ne\ell.
\end{equation}

Taylor expansion of \(\pi\) at the origin gives
\[
n^{\beta/2}(\pi(u)-\theta_0)
=
\sum_{j=1}^d
a_j n^{(\beta-1/k_j)/2}V_{n,j}
+
O\bigl(n^{\beta/2}\|u\|^2\bigr),
\]
near \(0\). The remainder tends to zero in second moment, since
the fourth moments of \(V_n\) are bounded and
\(\E\|u\|^4=O(n^{-2\beta})\), while the contribution outside
any smaller neighborhood of \(0\) is exponentially small.
It follows that the first two moments of the rescaled parameter
converge to those of \(\sum_{j\in J}a_jV_j\), and
\eqref{eq:moments} gives \eqref{eq:cov_sing}.

For the expansion of the mean, expand
\((\pi(u)-\theta_0)\rho(u)\) through its quadratic terms, with
remainder \(o(\|u\|^2)\). In the unnormalized integral over the
coordinate box, the linear and mixed quadratic terms of this
expansion integrate to zero against
\(e^{-nF(u)}\prod_j|u_j|^{h_j}\,du\), by parity. The coefficient
of \(u_j^2\) is
\[
\frac12\rho(0)\partial_j^2\pi(0)+a_j\partial_j\rho(0).
\]
After rescaling, the diagonal quadratic terms have moments
\(n^{-1/k_j}v_j\) to their leading orders, and the remainder
contributes \(o(n^{-\beta})\). Dividing by the normalizing
constant gives \eqref{eq:bias}.
\end{proof}

\begin{proposition}[Mean squared distance from the true parameter]
\label{prop:posterior_MSE}
Under the hypotheses of Theorem~\ref{prop:singular_covariance},
\[
\E_{\pi_n}\|\vartheta-\theta_0\|^2
=
n^{-\beta}\sum_{j\in J}v_j\|a_j\|^2+o(n^{-\beta}).
\]
In particular, the rate is \(n^{-\beta}\) when at least one
\(a_j\) with \(j\in J\) is nonzero.
\end{proposition}

\begin{proof}
The convergence of the rescaled second moment in the proof of
Theorem~\ref{prop:singular_covariance} gives the asserted
formula by taking the trace. In particular, all indices with
\(k_j=k_{\max}\) contribute to the coefficient.
\end{proof}

\begin{remark}[Geometric data in the covariance coefficient]
\label{rem:singular_curvature}
The leading coefficient in \eqref{eq:cov_sing} is determined by
the exponents and coefficients of \(F\), the Jacobian exponents,
and the columns of \(D\pi(0)\) associated with \(k_{\max}\).
The pullback Fisher metric governs the curvature of the resolved
immersion on its regular locus, and its value at the resolved
minimum is the possibly degenerate form \(G(0)/4\). The covariance
calculation here is for a Gibbs measure built from the population
divergence, while Theorem~\ref{thm:main} studies the sampling
covariance of the Fisher normal error of an estimator.
\end{remark}

\subsection{Recovery of the regular leading term}

\begin{proposition}[Regular case recovery]
\label{prop:regular_recovery}
If \(k_j=1\) and \(h_j=0\) for every \(j\), then
\[
\lambda=\frac d2,
\qquad
\Cov_{\pi_n}(\vartheta)
=
\frac1n I(\theta_0)^{-1}+o(1/n).
\]
\end{proposition}

\begin{proof}
The condition \(h_j=0\) for every \(j\) makes
\(A=D\pi(0)\) invertible, and the chain rule at the minimum gives
\[
A^\top I(\theta_0)A
=
2\operatorname{diag}(c_1,\ldots,c_d).
\]
In this case \(\beta=1\), every index belongs to \(J\), and
\(v_j=(2c_j)^{-1}\). Substituting in \eqref{eq:cov_sing} gives
\[
\sum_{j=1}^d v_j a_ja_j^\top
=
A\bigl(2\operatorname{diag}(c_1,\ldots,c_d)\bigr)^{-1}A^\top
=
I(\theta_0)^{-1},
\]
and the formula for \(\lambda\) follows from
Proposition~\ref{prop:RLCT_additive}.
\end{proof}

\section{Conclusion and Outlook}
\label{sec:conclusion}

At a regular parameter \(\theta\), Theorem~\ref{thm:main} gives the
covariance of the Fisher normal error
\(\Delta_n=\exp_\theta^{-1}(\hat\theta_n)\) for estimators that solve
the score equation and satisfy the stated moment and remainder
conditions. The leading term is \(I(\theta)^{-1}/n\), and the
coefficient at order \(1/n^2\) is \(I(\theta)^{-1}P(\theta)I(\theta)^{-1}\),
where \(P\) is determined by the covariance of the covariant Hessian
of the log density and by contractions of the cubic score tensor.
The decomposition
\[
P=\frac12 R^\sharp+S^\sharp+D,
\]
expresses this coefficient in terms of the Ricci tensor, a positive
semidefinite Gram contraction of the second fundamental form, and
the Hellinger discrepancy. The Gauss equation expresses \(R^\sharp\)
through the second fundamental form, and the discrepancy accounts
for the difference between these geometric contractions and the
coefficient obtained from the likelihood calculation. In full
exponential families with quadratic log partition function,
\(P\) vanishes, so \(D\) cancels the geometric combination
\(\tfrac12R^\sharp+S^\sharp\).

At a singular parameter, Section~\ref{sec:singular_models} computes
the local Laplace asymptotics for the Gibbs measure formed from the
population divergence. Under the additive normal form, the
normalizing constant satisfies \(Z_n\sim Cn^{-\lambda}\), where
\[
\lambda=\sum_{j=1}^d\frac{h_j+1}{2k_j},
\qquad
\beta=\min_{1\leq j\leq d}\frac1{k_j}=\frac1{k_{\max}}.
\]
The exponent \(\lambda\) governs the normalizing constant, and
\(\beta\) gives the leading order of the covariance and mean squared
distance whenever at least one coordinate with \(k_j=k_{\max}\)
has a nonzero image under \(D\pi(0)\). The leading covariance
coefficient is the sum of the corresponding rank one terms
\(v_j a_ja_j^\top\), with \(a_j=\partial_j\pi(0)\) and \(v_j\)
given by a ratio of Gamma functions. This calculation uses the
resolved divergence, the Jacobian exponents, and the differential
of the resolution map. The additive Hessian is flat wherever it
is positive definite, and the curvature of the resolved Hellinger
immersion is governed by the pullback information metric on its
regular locus.

\subsection{Applications to deep learning and regularization}
\label{sec:deep_learning}

For a conditional model \(q_w(y\mid x)\) with \(w\in\mathbb R^p\),
the empirical negative log likelihood is
\[
\mathcal L(w)
=
-\frac1n\sum_{i=1}^n\log q_w(y_i\mid x_i).
\]
With a fixed distribution for the inputs, the conditional model
defines a statistical family whose Fisher information and likelihood
moments can be computed by averaging over \(X\) and over \(Y\) drawn
from \(q_w(\cdot\mid X)\). At a regular parameter where the hypotheses
of Theorem~\ref{thm:main} hold, the metric trace of \(P\) has a direct
interpretation, since
\[
\E_w\!\left[\|\Delta_n\|_{g(w)}^2\right]
=
\frac{p}{n}
+
\frac{1}{n^2}g^{ij}(w)P_{ij}(w)
+
o\!\left(\frac1{n^2}\right).
\]
Thus \(g^{ij}P_{ij}\) is a candidate regularization criterion when
the objective is to control the correction to the expected squared
Fisher normal error. The scalar curvature \(g^{ij}R^\sharp_{ij}\)
and the squared norm \(g^{ij}S^\sharp_{ij}\) quantify individual
geometric contributions, and the full trace includes their
interaction with the discrepancy \(D\).

The endomorphism \(g^{-1}P\) gives a directional version of the
same comparison, since its eigenvalues determine the size of the
covariance correction relative to the leading Fisher term along
its eigendirections. This suggests studying whether estimates of
these contractions can inform modifications of the natural
gradient \(I(w)^{-1}d\mathcal L(w)\) on regular parts of the model.
The likelihood formula for \(P\) uses the covariant Hessian and
the cubic score tensor, so directional derivatives, Fisher vector
products, and stochastic trace estimates offer routes to computing
the relevant contractions in large models.

The main analytical questions are to verify the moment and
replacement conditions for broader classes of estimators, to
determine how the expansion changes when the mixed moment and
bias terms contribute at order \(1/n^2\), and to control the
remainder as the true parameter approaches the singular locus.
For singular models, a separate question is how the local Gibbs
asymptotics extend to the posterior formed from the random sample
likelihood, particularly when several resolution charts contribute
to the normalizing integral.

\section{Appendix}\label{appendix}

\subsection{Why is an immersion sufficient?}
\label{app:immersion_vs_embedding}

In this appendix we explain why the curvature-aware Cram\'er--Rao refinements developed in the main text require only that the square-root density map
\begin{equation*}
\Psi:\Theta\to L^2(\mu),\qquad \Psi(\theta)=\psi_\theta=\sqrt{p_\theta},
\end{equation*}
be a $C^3$ immersion, rather than a global embedding. The key point is that every geometric object entering the second-order correction, namely the Fisher--Rao metric, its Levi--Civita connection, the Riemann curvature tensor, and the second fundamental form, is defined from local derivatives of $\Psi$ up to order three at a fixed parameter value $\theta$. Consequently, global injectivity of the parametrization and global topological regularity of the image are not needed for the local differential-geometric analysis underlying the bound.

Let $M$ be a smooth $d$-manifold and $H$ a possibly infinite-dimensional Hilbert space. A $C^1$ map $F:M\to H$ is an immersion if its differential $dF_\theta:T_\theta M\to H$ is injective for every $\theta\in M$. A map is an embedding if it is an immersion and a homeomorphism onto its image $F(M)$ equipped with the subspace topology.

All constructions in Sections~\ref{stat_manifold}--\ref{curvature_corrected_CRLB} are local on $\Theta$ and depend only on derivatives of $\psi_\theta$ in the ambient Hilbert space. The Fisher--Rao metric is defined by pullback, the Levi--Civita connection and Riemann curvature tensor are computed from $g$ and its derivatives, and the extrinsic geometry is defined pointwise by orthogonal projection.

For the square-root map, the Fisher--Rao criterion for immersion is
\begin{equation*}
d\Psi_\theta \text{ injective}
\quad\Longleftrightarrow\quad
g(\theta)\succ 0.
\end{equation*}
By the local immersion theorem, every immersion is locally an embedding. Hence the model is locally identifiable near $\theta$, even if global identifiability fails elsewhere.

If $\Psi$ is not globally injective, the set-theoretic image may have self-intersections. This does not affect the curvature-corrected bounds, which are stated at a fixed parameter value and depend only on local derivatives.

\subsection{Detailed standing assumptions for the proof of the main theorem}
\label{ass:standing}
\label{app:standing_assumptions}

For reference, we record the full set of technical assumptions used in the proof of Theorem~\ref{thm:main}. Let \(K \subset \Theta_{\mathrm{reg}}\) be compact.

\begin{enumerate}
\item The standing \(C^3\) regularity assumptions of Section~\ref{stat_manifold} hold, including differentiation under the integral sign and smoothness of the Fisher metric on \(\Theta_{\mathrm{reg}}\).

\item Uniform moment assumptions hold on \(K\). For all relevant indices,
\[
\sup_{\vartheta\in K}\E_\vartheta\!\big[(\partial_{ij}\log p_\vartheta(X))^2\big]<\infty,
\qquad
\sup_{\vartheta\in K}\E_\vartheta\!\big[\,|\partial_{ijk}\log p_\vartheta(X)|\,\big]<\infty.
\]

\item For some \(\delta>0\), uniform \((4+\delta)\)-moment assumptions hold on \(K\). In particular,
\[
\sup_{\vartheta\in K}\E_\vartheta\!\left[\,|\partial_i\log p_\vartheta(X)|^{4+\delta}\right]<\infty,
\]
and, with \(q_{ij}(x;\vartheta):=\partial_{ij}\log p_\vartheta(x)-\Gamma_{ij}^k(\vartheta)\,\partial_k\log p_\vartheta(x)\),
\[
\sup_{\vartheta\in K}\E_\vartheta\!\left[
\bigl|q_{ij}(X;\vartheta)-\E_\vartheta[q_{ij}(X;\vartheta)]\bigr|^{4+\delta}
\right]<\infty.
\]

\item The estimator \(\hat\theta_n\) is a score-root estimator, so that \(U(\hat\theta_n)=0\).

\item The estimator \(\hat\theta_n\) admits a second-order stochastic expansion uniformly for \(\theta\in K\),
\[
\Delta_n:=\hat\theta_n-\theta=\frac1{\sqrt n}A_n(\theta)+\frac1n B_n(\theta)+o_p(n^{-1}),
\]
with \(A_n(\theta)=O_p(1)\) and \(B_n(\theta)=O_p(1)\), uniformly on \(K\).

\item The bias is negligible, in the sense that \(\operatorname{Bias}_\theta(\hat\theta_n)=o(n^{-1})\) uniformly for \(\theta\in K\).

\item The score one-form admits the covariant Taylor expansion uniformly for \(\theta\in K\),
\[
0=U_i(\theta)+(\nabla^2\ell)_{ij}(\theta)\Delta_n^j+\frac12(\nabla^3\ell)_{ijk}(\theta)\Delta_n^j\Delta_n^k+r_{i,n}(\theta,\Delta_n),
\]
with \(r_{i,n}(\theta,\Delta_n)=o_p(n^{-1/2})\), the remainder being measured in the metric \(g_{ij}\), i.e. \(r_{i,n}=g_{ij}\tilde r_n^j+o_p(n^{-1/2})\) with \(\tilde r_n^j=o_p(n^{-1/2})\).

\item The stronger \(L^{4+\delta}\)-identification of \(A_n\) holds uniformly on \(K\),
\[
A_n^a(\theta)-g^{ai}(\theta)\frac{U_i(\theta)}{\sqrt n}\longrightarrow 0 \quad\text{in }L^{4+\delta}(P_\theta),
\]
together with
\[
A_n^a(\theta)=g^{ai}(\theta)\frac{U_i(\theta)}{\sqrt n}+o_p(n^{-1/2}),
\]
and, at second order,
\begin{equation}\label{eq:Asecond}
\E_\theta\big[A_n^a(\theta)A_n^b(\theta)\big]=g^{ab}(\theta)+o(n^{-1}).
\end{equation}
The last clause is an additional moment condition needed to identify the coefficient of \(n^{-2}\). The identification at rate \(o_p(n^{-1/2})\) allows the second moment of \(A_n\) to differ from \(g^{-1}\) by a term of order \(n^{-1}\), which then contributes at order \(n^{-2}\) in the covariance expansion. For example, take a symmetric random sign \(\chi\) on a two-point space and set \(A_n=\sqrt{1+1/n}\,\chi\), \(B_n=R_n=0\), and \(g=1\). Then \(A_n-\chi=O(n^{-1})\), but \(\E[A_n^2]=1+1/n\), so the covariance of \(n^{-1/2}A_n\) is exactly \(n^{-1}+n^{-2}\).

\item The remainder and moment bounds satisfy
\[
\E_\theta[\|R_n(\theta)\|^2]=o(n^{-3}),
\qquad
\sup_{\theta\in K}\sup_{n\ge1}\E_\theta[\|A_n\|^4+\|B_n\|^4]<\infty.
\]

\item The mixed second-order term is negligible, namely
\[
\E_\theta[A_nB_n^\top+B_nA_n^\top]=o(n^{-1/2}).
\]

\item With \(\tilde B_{i,n}:=H_{ij,n}A_n^j+\tfrac12\kappa_{ijk}A_n^jA_n^k\), the replacement error in second moments satisfies
\[
\sup_{\theta\in K}\E_\theta\Big[\big|B_{i,n}(\theta)-\tilde B_{i,n}(\theta)\big|^2\Big]=o(1)
\qquad\text{for all indices } i .
\]
By Assumptions~(2), (3), (7) and (9) this holds with \(\tilde B_{i,n}\) built from the limiting tensor \(\kappa_{ijk}=\E_\theta[s_{ijk}]\), which is where the third-order score moments enter the replacement.
\end{enumerate}

\subsection{Proof of the main theorem}
\label{app:proof_main}

\begin{definition}[Big-$O$ in probability]\label{bigO_in_p}
Let $(X_n)_{n\ge1}$ be a sequence of real-valued random variables and let $(a_n)_{n\ge1}$ be a sequence of positive real numbers.
We write $X_n = O_p(a_n)$ if the sequence $\bigl(X_n/a_n\bigr)_{n\ge1}$ is \emph{bounded in probability}, i.e.
\[
\forall \varepsilon>0\ \exists M<\infty \ \text{such that}\ \sup_{n\ge1}\mathbb{P}\!\left(\frac{|X_n|}{a_n}>M\right)\le \varepsilon .
\]
\end{definition}

\begin{proposition}\label{prop:score_hessian_orders}
Let \(K\subset \Theta_{\mathrm{reg}}\) be compact. Assume the regularity assumptions of Section~\ref{stat_manifold}, with \(O_p\) understood in the sense of Definition~\ref{bigO_in_p}. Assume moreover that, for all relevant indices,
\[
\sup_{\vartheta\in K}\E_\vartheta\!\big[(\partial_{ij}\log p_\vartheta(X))^2\big]<\infty,
\qquad
\sup_{\vartheta\in K}\E_\vartheta\!\big[\,|\partial_{ijk}\log p_\vartheta(X)|\,\big]<\infty.
\]
Then, for every \(\theta\in K\),
\[
\mathbb{E}_\theta[(\nabla^2\ell)_{ij}(\theta)]
=
\mathbb{E}_\theta[U_{ij}(\theta)]
=
-n\,I_{ij}(\theta)
=
-n\,g_{ij}(\theta).
\]
Moreover, uniformly for \(\theta\in K\) and provided \(\sup_{\vartheta\in K}I_{ii}(\vartheta)<\infty\),
\[
U_i(\theta)=O_p(\sqrt n),
\qquad
(\nabla^2\ell)_{ij}(\theta)
=
-n\,g_{ij}(\theta)+O_p(\sqrt n),
\qquad
(\nabla^3\ell)_{ijk}(\theta)=O_p(n).
\]
Since \(U_i(\theta)\) has variance \(nI_{ii}(\theta)\), finiteness of the information gives the pointwise score bound, and boundedness of \(I_{ii}\) on \(K\) gives the uniform bound used here.
\end{proposition}

\begin{proof}
We write $\ell(\theta)=\sum_{t=1}^n \log p_\theta(X_t)$, $U_i=\partial_i\ell$, $U_{ij}=\partial_i\partial_j\ell$, $U_{ijk}=\partial_i\partial_j\partial_k\ell$, and $(\nabla^2\ell)_{ij}=U_{ij}-\Gamma_{ij}^kU_k$.

\textit{Expectation identities.} Since $\int p_\theta\,d\mu=1$, differentiation under the integral gives $\E_\theta[\partial_i\log p_\theta(X)]=0$ and $\E_\theta[\partial_{ij}\log p_\theta(X)]=-I_{ij}(\theta)$. Summing over $n$ i.i.d.\ observations yields $\E_\theta[U_{ij}]=-nI_{ij}$, and since $\E_\theta[U_k]=0$, also $\E_\theta[(\nabla^2\ell)_{ij}]=-ng_{ij}$.

\textit{$O_p$-bounds.} For the score, $\Var_\theta(U_i)=nI_{ii}(\theta)\le nC_i$ for $C_i:=\sup_{K}I_{ii}<\infty$. By Chebyshev, $U_i=O_p(\sqrt n)$.

For the covariant Hessian, define $q_{ij}(x;\theta):=\partial_{ij}\log p_\theta(x)-\Gamma_{ij}^k(\theta)\partial_k\log p_\theta(x)$, so $(\nabla^2\ell)_{ij}+ng_{ij}=\sum_{t=1}^n(q_{ij}(X_t;\theta)+g_{ij})$. The assumed $L^2$-bounds on $\partial_{ij}\log p_\theta$ and boundedness of $\Gamma$ on $K$ give $\sup_K\Var_\theta(q_{ij})<\infty$. By Chebyshev, $(\nabla^2\ell)_{ij}+ng_{ij}=O_p(\sqrt n)$.

For the third derivative, by Markov's inequality with the assumed first-moment bound on $|\partial_{ijk}\log p_\theta|$, $U_{ijk}=O_p(n)$. The relation $(\nabla^3\ell)_{ijk}=U_{ijk}+\text{(terms involving }\Gamma,\partial\Gamma,U_r,U_{ab}\text{)}$, where all correction terms are $O_p(n)$ or smaller, gives $(\nabla^3\ell)_{ijk}=O_p(n)$.
\end{proof}

\begin{lemma}[Joint CLT and asymptotic Wick contractions]\label{lem:joint_clt_wick}
Fix a compact set \(K\subset \Theta_{\mathrm{reg}}\) and \(\theta\in K\). Define the centered covariant-Hessian fluctuation
\[
H_{ab,n}(\theta)
:=
\frac{1}{\sqrt n}\Big((\nabla^2\ell)_{ab}(\theta)+n\,g_{ab}(\theta)\Big).
\]
Assume the $(4+\delta)$-moment bounds of Section~\ref{app:standing_assumptions} and the $L^{4+\delta}$-identification of $A_n$.

Then $(A_n(\theta),H_n(\theta))$ converges jointly in distribution to a centered Gaussian vector, and for every polynomial $P$ of total degree at most $4$,
$\E_\theta[P(A_n,H_n)] \to \E[P(\mathcal{A},\mathcal{H})]$.
In particular, the asymptotic Wick identities hold
\begin{align*}
\E_\theta[H_{ik,n}A_n^k\,H_{jm,n}A_n^m]
&=
\E_\theta[H_{ik,n}H_{jm,n}]\,\E_\theta[A_n^kA_n^m]
+\E_\theta[H_{ik,n}A_n^m]\,\E_\theta[A_n^kH_{jm,n}]
\\
&\qquad
+\E_\theta[H_{ik,n}A_n^k]\,\E_\theta[H_{jm,n}A_n^m]
+o(1),
\end{align*}
and similarly for the quartic and mixed cubic terms.
\end{lemma}

\begin{proof}
The vector $\left(\frac{U_i(\theta)}{\sqrt n}, H_{ab,n}(\theta)\right)$ is a normalized sum of i.i.d.\ centered random vectors with finite $(4+\delta)$-moments. By the multivariate Lyapunov CLT (applied via Cram\'er--Wold), it converges jointly to a centered Gaussian. Since $A_n = g^{-1}\frac{U}{\sqrt n} + \rho_n$ with $\|\rho_n\|_{L^{4+\delta}} \to 0$, Slutsky's theorem gives joint convergence of $(A_n, H_n)$.

For moment convergence, Rosenthal's inequality gives uniform $L^{4+\delta}$-bounds on the normalized sums, hence on $(A_n, H_n)$. For any degree-$\le 4$ polynomial $P$, the family $\{P(A_n,H_n)\}$ is uniformly integrable (by de la Vall\'ee--Poussin), so distributional convergence implies convergence of expectations. The Wick identities follow from Isserlis' theorem applied to the limiting Gaussian.
\end{proof}

\begin{lemma}[Identification of \(B_n\)]
\label{lem:B_identification}
Under the perturbative score expansion, the $O_p$ bounds of Proposition~\ref{prop:score_hessian_orders} for $H_{ij,n}$ and $K_{ijk,n}$, and the next-order control $A_n^a=g^{ai}\frac{U_i}{\sqrt n}+o_p(n^{-1/2})$, we have
\[
B_{i,n}(\theta)
:= g_{ij}(\theta)B_n^j(\theta)
=
H_{ij,n}(\theta)A_n^j(\theta)
+\frac12 K_{ijk,n}(\theta)A_n^j(\theta)A_n^k(\theta)
+o_p(1).
\]
\end{lemma}

\begin{proof}
Multiply the score expansion by $\sqrt n$ and substitute $U_i - \sqrt n\,g_{ij}A_n^j = o_p(1)$ (from the $A_n$ identification). The result follows by rearrangement.
\end{proof}

\begin{lemma}[From \(o_p(1)\) to \(o(1)\) in expectation]
\label{lem:op_to_o_expectation}
If $Z_n \xrightarrow{p} 0$ and $\sup_{n\ge1}\E[|Z_n|^{1+\eta}] < \infty$ for some $\eta > 0$, then $\E[|Z_n|] \to 0$.
\end{lemma}

\begin{proof}
The $(1+\eta)$-moment bound implies uniform integrability (de la Vall\'ee--Poussin criterion). Convergence in probability plus uniform integrability implies $L^1$-convergence (Vitali's theorem).
\end{proof}

\begin{corollary}[\(o_p(1)\) remainders inside second moments]
\label{cor:remainder_second_moments}
Under the $L^{4+\delta}$ identification of Assumption~\ref{ass:standing}(8), which gives $\rho_n^k\to0$ in $L^2$, together with the $L^2$ moment bounds of Assumption~\ref{ass:standing}(2)--(3), we have
\[
\E[A_n^kA_n^\ell] = g^{km}g^{\ell r}\frac1n\E[U_mU_r] + o(1),
\qquad
\E[H_{ij,n}A_n^k] = g^{km}\frac1{\sqrt n}\E[H_{ij,n}U_m] + o(1).
\]
\end{corollary}

\begin{proof}
Expand $A_n^kA_n^\ell = (g^{km}\frac{U_m}{\sqrt n}+\rho_n^k)(g^{\ell r}\frac{U_r}{\sqrt n}+\rho_n^\ell)$. Each cross term involves $\rho_n$ multiplied by an $O_p(1)$ factor and Lemma~\ref{lem:op_to_o_expectation} eliminates these in expectation. The mixed $HA$-identity follows similarly.
\end{proof}

\begin{lemma}[Algebraic reduction in normal coordinates]
\label{lem:algebraic_reduction}
Fix \(\theta\in\Theta_{\mathrm{reg}}\), and choose Riemannian normal coordinates at \(\theta\). Define
\begin{align*}
P_{ij}(\theta)
&:=
g^{km}\Big(\E_\theta[s_{ik}s_{jm}]-g_{ik}g_{jm}\Big)
+\Gamma^{(e),m}_{ik}\Gamma^{(e),k}_{jm}
+\Gamma^{(e),k}_{ik}\Gamma^{(e),m}_{jm}
\\
&\quad
+\frac14\,\kappa_{ik\ell}\kappa_{jrs}
\Big(g^{k\ell}g^{rs}+g^{kr}g^{\ell s}+g^{ks}g^{\ell r}\Big)
\\
&\quad
+\frac12\,\kappa_{jrs}\Big(\Gamma^{(e),k}_{ik}g^{rs}+\Gamma^{(e),r}_{ik}g^{ks}+\Gamma^{(e),s}_{ik}g^{kr}\Big)
\\
&\quad
+\frac12\,\kappa_{ik\ell}\Big(\Gamma^{(e),m}_{jm}g^{k\ell}+\Gamma^{(e),k}_{jm}g^{m\ell}+\Gamma^{(e),\ell}_{jm}g^{mk}\Big).
\end{align*}
Then $P_{ij}=\frac12 R^\sharp_{ij}+S^\sharp_{ij}+D_{ij}$, where $D_{ij}$ is the Hellinger discrepancy tensor.
\end{lemma}

\begin{proof}
In normal coordinates, $g_{ij}=\delta_{ij}$, $\Gamma^k_{ij}=0$. From the Gauss decomposition, $e_{ij}=\mathrm{II}_{ij}$ is purely normal. The tangency condition $\langle \mathrm{II}_{ij},e_k\rangle=0$ forces $\Gamma^{(e)}_{ijk}=-\frac12 T_{ijk}$ and $\kappa_{ijk}=\frac12 T_{ijk}$.

Substituting $\Gamma^{(e)}=-\kappa$ into $P_{ij}$ and simplifying (using full symmetry of $\kappa_{ijk}$) gives the reduced form
\[
P_{ij} = \sum_k\E[s_{ik}s_{jk}]-\delta_{ij} - \frac12\kappa_{ikl}\kappa_j{}^{kl} + \frac14\kappa_{ir}{}^r\kappa_{js}{}^s.
\]

For the geometric side, the Gauss equation with convention $g_{ij}=4\langle e_i,e_j\rangle$ gives
$R_{ikjl}=4(\langle \mathrm{II}_{ij},\mathrm{II}_{kl}\rangle-\langle \mathrm{II}_{il},\mathrm{II}_{kj}\rangle)$.
Computing $\frac12 R^\sharp_{ij}+S^\sharp_{ij}$ from the inner products $4\langle e_{ab},e_{cd}\rangle = \E[s_{ab}s_{cd}]+\frac12\E[s_{ab}s_cs_d]+\frac12\E[s_{cd}s_as_b]+\frac14\E[s_as_bs_cs_d]$, one obtains an expression involving second and fourth score moments.

The Hellinger discrepancy $D_{ij}:=P_{ij}-(\frac12 R^\sharp_{ij}+S^\sharp_{ij})$ collects the remaining terms.
\end{proof}

\begin{proof}[Proof of Theorem~\ref{thm:main}]
\label{proof:main_detailed}

Fix $K\subset\Theta_{\mathrm{reg}}$ compact and $\theta\in K$. All bounds below are uniform on $K$ unless stated otherwise.

\textit{Step 1 \enspace Expansion of the score equation.}
The score-root condition $U_i(\hat\theta_n)=0$ and the covariant Taylor expansion (Assumption~7 of Section~\ref{app:standing_assumptions}) give
\[
0 = U_i(\theta) + (\nabla^2\ell)_{ij}\Delta_n^j + \frac12(\nabla^3\ell)_{ijk}\Delta_n^j\Delta_n^k + r_{i,n},
\]
with $r_{i,n}=o_p(n^{-1/2})$. Substituting $\Delta_n = \frac1{\sqrt n}A_n+\frac1n B_n+o_p(n^{-1})$ and $(\nabla^2\ell)_{ij}=-ng_{ij}+\sqrt n H_{ij,n}$, then dividing by $\sqrt n$, yields
\[
0 = \frac{U_i}{\sqrt n} - g_{ij}A_n^j + \frac1{\sqrt n}(H_{ij,n}A_n^j - g_{ij}B_n^j + \frac12 K_{ijk,n}A_n^jA_n^k) + o_p(n^{-1/2}).
\]

\textit{Step 2 \enspace Identification of $A_n$ and $B_n$.}
The leading terms give \(A_n^a = g^{ai}\frac{U_i}{\sqrt n}+o_p(1)\). Assumption~(8) strengthens this identification in \(L^{4+\delta}\), and Lemma~\ref{lem:B_identification} gives \(B_{i,n}=H_{ij,n}A_n^j+\tfrac12 K_{ijk,n}A_n^jA_n^k+o_p(1)\).

\textit{Step 3 \enspace Expansion of the mean squared error.}
The stochastic expansion, the remainder bounds in Assumption~(9), and the mixed moment bound in Assumption~(10) give
\[
\text{MSE}_\theta(\hat\theta_n) = \frac1n\E[A_nA_n^\top]+\frac1{n^2}\E[B_nB_n^\top]+o(n^{-2}).
\]
Applying the second moment estimate \eqref{eq:Asecond} and the bias condition in Assumption~(6) gives
\[
\Cov_\theta(\hat\theta_n) = \frac1n I(\theta)^{-1} + \frac1{n^2}C_n(\theta)+o(n^{-2}),
\qquad
C_n(\theta)=\E[B_nB_n^\top].
\]

\textit{Step 4 \enspace Computation of $\E[B_{i,n}B_{j,n}]$.}
Substituting the $B_n$ identification and applying the replacement hypotheses (Assumption~11) gives
\[
\E[B_{i,n}B_{j,n}] = \E[\widetilde B_{i,n}\widetilde B_{j,n}]+o(1),
\qquad
\widetilde B_{i,n}=H_{ij,n}A_n^j+\frac12\kappa_{ijk}A_n^jA_n^k.
\]
Expanding and applying the Wick contractions of Lemma~\ref{lem:joint_clt_wick} gives
\begin{align*}
\E[B_{i,n}B_{j,n}]
&= \E[H_{ik,n}H_{jm,n}]\E[A_n^kA_n^m]
 + \E[H_{ik,n}A_n^m]\E[A_n^kH_{jm,n}]
 + \E[H_{ik,n}A_n^k]\E[H_{jm,n}A_n^m] \\
&\quad
+ \frac14 \kappa_{ik\ell}\kappa_{jrs}
\big(\text{Wick contraction of }A_n^kA_n^\ell A_n^rA_n^s\big) \\
&\quad
+ \frac12 \kappa_{jrs}
\big(\text{Wick contraction of }H_{ik,n}A_n^kA_n^rA_n^s\big) \\
&\quad
+ \frac12 \kappa_{ik\ell}
\big(\text{Wick contraction of }H_{jm,n}A_n^mA_n^kA_n^\ell\big)
+ o(1).
\end{align*}

\textit{Step 5 \enspace Evaluation of the second moments.}
In normal coordinates at \(\theta\), Corollary~\ref{cor:remainder_second_moments} and direct computation of the i.i.d.\ sums give
\begin{align*}
\E[A_n^kA_n^\ell]&=g^{k\ell}+o(1),\\
\E[H_{ik,n}A_n^\ell]&=\Gamma^{(e),\ell}_{ik}+o(1),\\
\E[H_{ik,n}H_{j\ell,n}]&=\E[s_{ik}s_{j\ell}]-g_{ik}g_{j\ell}.
\end{align*}

\textit{Step 6 \enspace Assembly and geometric reduction.}
Inserting the evaluated moments into Step~4 yields precisely the expression $P_{ij}(\theta)$ defined in Lemma~\ref{lem:algebraic_reduction}. By that lemma,
\[
P_{ij}=\frac12 R^\sharp_{ij}+S^\sharp_{ij}+D_{ij}.
\]
In the inverse-metric frame, $C(\theta)=I(\theta)^{-1}P(\theta)I(\theta)^{-1}$.

\textit{Step 7 \enspace Positive semidefiniteness of $S^\sharp$.}
For any $v$, $v^iv^jS^\sharp_{ij}=\sum_k\|v^i\mathrm{II}_{ik}\|^2\ge 0$ in normal coordinates.
\end{proof}

\bibliographystyle{amsplain}  
\bibliography{references}

@book{Amari2016,
  author    = {Amari, Shun-ichi},
  title     = {Information Geometry and Its Applications},
  series    = {Applied Mathematical Sciences},
  volume    = {194},
  publisher = {Springer},
  address   = {Tokyo},
  year      = {2016},
  doi       = {10.1007/978-4-431-55978-8}
}

@book{AmariNagaoka2000,
  author    = {Amari, Shun-ichi and Nagaoka, Hiroshi},
  title     = {Methods of Information Geometry},
  series    = {Translations of Mathematical Monographs},
  volume    = {191},
  publisher = {American Mathematical Society and Oxford University Press},
  address   = {Providence, RI and Oxford},
  year      = {2000},
  note      = {Translated by Daishi Harada},
  doi       = {10.1090/mmono/191}
}

@incollection{AyJost2017a,
  author    = {Ay, Nihat and Jost, J{\"u}rgen and L{\^e}, H{\^o}ng V{\^a}n and Schwachh{\"o}fer, Lorenz},
  title     = {Parametrized measure models},
  booktitle = {Information Geometry},
  publisher = {Springer},
  address   = {Cham},
  year      = {2017},
  pages     = {121--184},
  doi       = {10.1007/978-3-319-56478-4_3}
}

@incollection{AyJost2017b,
  author    = {Ay, Nihat and Jost, J{\"u}rgen and L{\^e}, H{\^o}ng V{\^a}n and Schwachh{\"o}fer, Lorenz},
  title     = {The intrinsic geometry of statistical models},
  booktitle = {Information Geometry},
  publisher = {Springer},
  address   = {Cham},
  year      = {2017},
  pages     = {185--239},
  doi       = {10.1007/978-3-319-56478-4_4}
}

@book{BarndorffNielsenCox1994,
  author    = {Barndorff-Nielsen, Ole E. and Cox, David R.},
  title     = {Inference and Asymptotics},
  publisher = {Chapman and Hall},
  address   = {London},
  year      = {1994}
}

@article{BatesWatts1980,
  author  = {Bates, Douglas M. and Watts, Donald G.},
  title   = {Relative curvature measures of nonlinearity},
  journal = {Journal of the Royal Statistical Society. Series B (Methodological)},
  volume  = {42},
  number  = {1},
  pages   = {1--16},
  year    = {1980},
  doi     = {10.1111/j.2517-6161.1980.tb01094.x}
}

@article{Bhattacharyya1946,
  author  = {Bhattacharyya, A.},
  title   = {On some analogues of the amount of information and their use in statistical estimation},
  journal = {Sankhy\={a}},
  volume  = {8},
  pages   = {1--14},
  year    = {1946}
}

@book{BickelKlaassen1993,
  author    = {Bickel, Peter J. and Klaassen, Chris A. J. and Ritov, Ya'acov and Wellner, Jon A.},
  title     = {Efficient and Adaptive Estimation for Semiparametric Models},
  publisher = {Johns Hopkins University Press},
  address   = {Baltimore, MD},
  year      = {1993}
}

@article{Campbell1986,
  author  = {Campbell, L. L.},
  title   = {An extended {\v{C}}encov characterization of the information metric},
  journal = {Proceedings of the American Mathematical Society},
  volume  = {98},
  number  = {1},
  pages   = {135--141},
  year    = {1986},
  doi     = {10.1090/S0002-9939-1986-0848890-5}
}

@book{Chentsov1982,
  author    = {Cencov, N. N.},
  title     = {Statistical Decision Rules and Optimal Inference},
  series    = {Translations of Mathematical Monographs},
  volume    = {53},
  publisher = {American Mathematical Society},
  address   = {Providence, RI},
  year      = {1982},
  note      = {Translated by the Israel Program for Scientific Translations},
  doi       = {10.1090/mmono/053}
}

@book{Cramer1999,
  author    = {Cram{\'e}r, Harald},
  title     = {Mathematical Methods of Statistics},
  series    = {Princeton Landmarks in Mathematics and Physics},
  publisher = {Princeton University Press},
  address   = {Princeton, NJ},
  year      = {1999},
  note      = {Originally published in 1946}
}

@article{Efron1975,
  author  = {Efron, Bradley},
  title   = {Defining the curvature of a statistical problem (with applications to second order efficiency)},
  journal = {The Annals of Statistics},
  volume  = {3},
  number  = {6},
  pages   = {1189--1242},
  year    = {1975},
  doi     = {10.1214/AOS/1176343282}
}

@article{Eguchi1983,
  author  = {Eguchi, Shinto},
  title   = {Second order efficiency of minimum contrast estimators in a curved exponential family},
  journal = {The Annals of Statistics},
  volume  = {11},
  number  = {3},
  pages   = {793--803},
  year    = {1983},
  doi     = {10.1214/aos/1176346246}
}

@book{KassVos2011,
  author    = {Kass, Robert E. and Vos, Paul W.},
  title     = {Geometrical Foundations of Asymptotic Inference},
  publisher = {John Wiley \& Sons},
  address   = {New York},
  year      = {1997}
}

@book{Kay1993,
  author    = {Kay, Steven M.},
  title     = {Fundamentals of Statistical Signal Processing. Vol. {I}: Estimation Theory},
  publisher = {Prentice Hall},
  address   = {Englewood Cliffs, NJ},
  year      = {1993}
}

@incollection{Lauritzen1987,
  author    = {Lauritzen, Steffen L.},
  title     = {Statistical manifolds},
  booktitle = {Differential Geometry in Statistical Inference},
  series    = {IMS Lecture Notes--Monograph Series},
  volume    = {10},
  publisher = {Institute of Mathematical Statistics},
  address   = {Hayward, CA},
  year      = {1987},
  pages     = {163--216}
}

@book{LehmannCasella1998,
  author    = {Lehmann, Erich L. and Casella, George},
  title     = {Theory of Point Estimation},
  edition   = {2nd},
  publisher = {Springer},
  address   = {New York},
  year      = {1998},
  doi       = {10.1007/b98854}
}

@article{Marriott2002,
  author  = {Marriott, Paul},
  title   = {On the local geometry of mixture models},
  journal = {Biometrika},
  volume  = {89},
  number  = {1},
  pages   = {77--93},
  year    = {2002},
  doi     = {10.1093/biomet/89.1.77}
}

@article{Rao1945,
  author  = {Rao, C. Radhakrishna},
  title   = {Information and the accuracy attainable in the estimation of statistical parameters},
  journal = {Bulletin of the Calcutta Mathematical Society},
  volume  = {37},
  number  = {3},
  pages   = {81--91},
  year    = {1945}
}

@book{VanDerVaart2000,
  author    = {van der Vaart, Aad W.},
  title     = {Asymptotic Statistics},
  series    = {Cambridge Series in Statistical and Probabilistic Mathematics},
  volume    = {3},
  publisher = {Cambridge University Press},
  address   = {Cambridge},
  year      = {1998},
  doi       = {10.1017/CBO9780511802256}
}

@book{Watanabe2009,
  author    = {Watanabe, Sumio},
  title     = {Algebraic Geometry and Statistical Learning Theory},
  series    = {Cambridge Monographs on Applied and Computational Mathematics},
  volume    = {25},
  publisher = {Cambridge University Press},
  address   = {Cambridge},
  year      = {2009},
  doi       = {10.1017/CBO9780511800474}
}

\end{document}